\documentclass{article}
\usepackage{arxiv}
\usepackage{amsmath}
\usepackage{ amssymb}   
\usepackage{mathtools}

\usepackage{amssymb,amsmath,amsthm}
\newcommand{\bu}{\boldsymbol{u}}
\newcommand{\bv}{\boldsymbol{v}}

\newcommand{\w}{\boldsymbol{w}}
\newcommand{\vv}{\boldsymbol{V}}
\newcommand{\xx}{\boldsymbol{X}}

\newcommand{\ff}{\boldsymbol{f}}

\newcommand{\bx}{{\boldsymbol{ \xi}}}
\newcommand{\tht}{\vartheta_{h,\tau}}
\newcommand{\bxt}{{\boldsymbol{ \xi}}_{h,\tau}}
\newcommand{\tu}{\widetilde{\boldsymbol{u}}}
\newcommand{\tp}{\widetilde{p}}
\newcommand{\be}{{\boldsymbol{\eta}}}

\newcommand{\uht}{\boldsymbol{u}_{h,\tau}}
\newcommand{\eht}{\boldsymbol{e}_{h,\tau}}
\newcommand{\vht}{\boldsymbol{v}_{h,\tau}}
\newcommand{\pht}{p_{h,\tau}}
\newcommand{\qht}{q_{h,\tau}}
\newcommand{\rhoht}{\varrho_{h,\tau}}
\newcommand{\intin}{\int_{I_n}}
\DeclareMathOperator*{\Div}{div}
\newcommand{\unh}{\boldsymbol{U}_{n,h}^i}
\newcommand{\enh}{\boldsymbol{e}_{n,h}^i}
\newcommand{\pnh}{P_{n,h}^i}
\newcommand{\wnh}{\boldsymbol{W}_{n,h}^i}
\newcommand{\Qnh}{Q_{n,h}^i}

\newcommand{\bet}{\be_\tau}
\newcommand{\beh}{\be_h}
\newcommand{\phih}{\varphi_h}
\newcommand{\phit}{\varphi_\tau}

\DeclareMathOperator*{\esssup}{ess\;sup}
\newcommand{\Qn}[1]{Q_n\left[#1\right]}
\newcommand{\Q}[1]{Q\left[#1\right]}

\newtheorem{assumption}{\sc Assumption}

\usepackage{caption}
\usepackage{graphicx}
\usepackage{tabularx}
\usepackage{tikz}
\usepackage{pgfplots}
\pgfplotsset{compat=newest}
\usetikzlibrary{external}
\newtheorem{lemma}{Lemma}
\newtheorem{theorem}{Theorem}
\newtheorem{remark}{Remark}

\tikzsetfigurename{figure}

\tikzset{external/system call={pdflatex \tikzexternalcheckshellescape %
		-halt-on-error -interaction=batchmode -jobname "\image" "\texsource"}}

\pgfplotstableread{
	1 3.925652e-02 5.533806e-04 5.524994e-01 4.917787e-02
	2 1.042449e-02 2.905527e-04 2.900683e-01 1.243839e-02
	3 2.636872e-03 1.465544e-04 1.463150e-01 3.106751e-03
	4 6.566679e-04 7.330227e-05 7.318370e-02 7.748854e-04
	5 1.624598e-04 3.655575e-05 3.649650e-02 1.934728e-04
	6 3.993797e-05 1.813998e-05 1.811041e-02 4.834100e-05
}\dgonelps
\pgfplotstableread{
	1 1.189094e-02 1.516806e-01 1.517785e-01 6.417324e-02
	2 1.478285e-03 3.777805e-02 3.779240e-02 1.270786e-02
	3 1.827262e-04 9.431538e-03 9.433407e-03 3.026147e-03
	4 2.275983e-05 2.357224e-03 2.357461e-03 7.513633e-04
	5 2.842520e-06 5.892704e-04 5.893031e-04 1.876253e-04
	6 3.569349e-07 1.473156e-04 1.473319e-04 4.689521e-05
}\dgonegal
\pgfplotstableread{
	1 0.80075     0.0125036    11.9433   7.58505
	2 0.201945    0.00127737   2.11282   2.66637
	3 0.0536967   0.000182485  0.682657  0.72417
	4 0.0136486   3.09163e-05  0.239197  0.184778
	5 0.00342651  6.44192e-06  0.084381  0.0464301
	6 0.000857528 1.51812e-06  0.0297856 0.0116223
	7 0.000214438 3.73492e-07  0.0105197 0.00290648
}\dgonetime
\pgfplotstableread{
	1 0.921185    0.0136931    1.16108      9.00701
	2 0.133781    0.00159433   0.190087     2.83294
	3 0.0178743   0.000195998  0.023616     0.746964
	4 0.00227272  2.41903e-05  0.00292556   0.189213
	5 0.000285265 2.99902e-06  0.000364299  0.0474554
	6 3.56927e-05 3.73212e-07  4.54554e-05  0.011873
	7 4.46253e-06 4.65453e-08  5.67688e-06  0.00296876
}\dgonetimepp
\pgfplotstableread{
	1 0.267571000 2.078810000 2.729970000 3.723160000
	2 0.040826300 0.140675000 0.460916000 0.599011000
	3 0.005399120 0.012339700 0.083914400 0.079572300
	4 0.000684525 0.001294990 0.014915500 0.010099200
	5 8.58708e-05 0.000152365 0.002635100 0.001267210
	6 1.07434e-05 1.87243e-05 0.000465231 0.000158552
	7 1.34323e-06 2.33021e-06 8.21681e-05 1.98237e-05
}\dgtwotime
\pgfplotstableread{
	1 0.253925000 2.718010000 0.338159000 6.349600000
	2 0.019690200 0.183688000 0.026263800 0.998386000
	3 0.001302980 0.011954400 0.001778380 0.132040000
	4 8.26061e-05 0.000744194 0.000112105 0.016737100
	5 5.18154e-06 4.61827e-05 6.98717e-06 0.002099460
	6 3.24146e-07 2.87367e-06 4.35470e-07 0.000262662
	7 2.02640e-08 1.79175e-07 2.71731e-08 3.28400e-05
}\dgtwotimepp
\pgfplotstableread{
	1 3.357335e-03 4.128991e-04 5.551823e-04 1.198345e-04 
	2 8.253489e-04 1.735094e-04 1.020201e-04 5.047504e-05 
	3 2.060866e-04 7.297079e-05 1.293472e-05 2.122571e-05 
	4 5.157912e-05 3.068436e-05 1.622838e-06 8.924378e-06 
	5 1.289948e-05 1.290169e-05 2.032927e-07 3.752241e-06 
	6 3.225202e-06 5.435888e-06 2.811472e-08 1.594091e-06 
	7 8.065224e-07 2.429374e-06 1.859074e-08 1.061616e-06 
}\rpdgonedgtwoltwoup

\usepackage[utf8]{inputenc} 
\usepackage[T1]{fontenc}    
\usepackage{hyperref}       
\usepackage{url}            
\usepackage{booktabs}       
\usepackage{amsfonts}       
\usepackage{nicefrac}       
\usepackage{microtype}      
\usepackage{lipsum}		

\title{Higher-order discontinuous Galerkin time discretizations 
	for the evolutionary Navier--Stokes equations}

\date{October 19, 2019}	

\author{
  Naveed Ahmed\\
  Department of Mathematics \& Natural Sciences\\
  Gulf University for Science \& Technology\\
  Block 5, Building 1,\\Mubarak Al-Abdullah Area/West Mishref
  Kuwait\\
  \texttt{ahmed.n@gust.edu.kw} \\
   \And
 Gunar Matthies \\
  Technische Universit\"at Dresden,\\
  Institut f\"ur Numerische Mathematik,\\
  01062 Dresden, Germany\\
  \texttt{gunar.matthies@tu-dresden.de} \\
}

\renewcommand{\headeright}{}
\renewcommand{\undertitle}{}

\begin{document}
\maketitle

\begin{abstract}
	Discontinuous Galerkin methods of higher order are applied as temporal
	discretizations for the transient Navier--Stokes equations. The spatial
	discretization based on inf-sup stable pairs of finite element spaces is
	stabilised using a one-level local projection stabilisation method.
	Optimal error bounds for the velocity with constants independent of the
	viscosity parameter are obtained for the semi-discrete case. For the
	fully discrete case, error estimates for both velocity and pressure are
	given. Numerical results support the theoretical predictions.
\end{abstract}


\section{Introduction}
\label{sec:intro}
Time-dependent flows of incompressible fluids can be described using the
transient incompressible Navier--Stokes equations that read in
dimensionless form
\begin{equation}\label{eq:nse}
\begin{alignedat}{2}
\bu'-\nu \Delta \bu +\bu \cdot \nabla \bu + \nabla p &= \ff &\quad
&\text{in } (0,T]\times \Omega, \\
\nabla \cdot \bu &= 0 &\quad &\text{in } (0,T]\times \Omega,
\end{alignedat} 
\end{equation}
where $\Omega\subset \mathbb{R}^d$, $d\in\{2,3\}$, denotes a domain with
Lipschitz boundary $\Gamma$ and $I=[0,T]$ a finite time interval with final
time $T>0$. Moreover, $\ff$ is a given body force, $\nu$ the viscosity,
$\bu$ and $p$ the velocity field and the pressure, respectively. The prime
denotes the time derivative of $\bu$ in a suitable sense. 
System~\eqref{eq:nse} of partial differential equations has to be closed
with an appropriate initial condition for the velocity at $t=0$ and
boundary conditions for the velocity on $(0,T]\times\Gamma$. For simplicity
of presentation, we consider the Navier--Stokes equations~\eqref{eq:nse}
equipped with homogeneous Dirichlet boundary conditions.

The analysis of transient incompressible Navier--Stokes equations is still
a great challenge in numerical analysis. There are several severe problems
that make the theoretical investigations demanding. Since the unique
solvability in three space dimensions is still open, regularity assumptions
are usually made. Furthermore, the handling of the nonlinear convection
term leads in general to an exponential grows of error bounds resulting
from an application of Gronwall's lemma.

In order to solve~\eqref{eq:nse} numerically, discretizations in space and
time are needed. We will consider in this paper temporal discretization by
discontinuous Galerkin (dG) methods of arbitrary order $k\ge 0$ and
spatial discretizations based on inf-sup stable pairs of finite
element spaces of order $r\ge 2$.

Since we are interesting in the convection-dominated case ($\nu\ll 1$),
stabilisation of the spatial discretization become necessary.
A very popular stabilisation
technique is the streamline upwind Petrov--Galerkin (SUPG) method,
see~\cite{HB79,BH82}, that is in most cases combined with the
pressure-stabilisation Petrov--Galerkin (PSPG) method,
see~\cite{HFB86,JS86,Tez92,John-Novo-PSPG}. The application of SUPG and
PSPG leads to an nonphysical coupling between velocity and pressure.
Furthermore, the non-symmetry of the stabilization may introduce additional
difficulties. Applied to time-dependent incompressible flow problems, the
drawbacks of SUPG and PSPG are more severe since much more terms have to
assembled, see~\cite{JS08,JN11}. We refer to~\cite{BBJL07} for a detailed
discussion on SUPG/PSPG methods.

We will consider the local projection stabilization (LPS) method to account
for dominating convection. Originally proposed for the Stokes equations
in~\cite{BB01}, the LPS methods was successfully transferred to transport
problems in~\cite{BB04}. LPS methods have been applied to the stationary Oseen
equations by~\cite{BB06,MST07,MT15} and convection-diffusion
problems by~\cite{MST08,AMTX11,AM15,BJK13}. The stabilising effect of LPS
schemes is based on a projection into a discontinuous space and results in
an additional control on the fluctuations of the gradient or parts of it.
Although LPS methods are weakly consistent only, the consistency error can
be bounded in an optimal way. The first LPS methods were two-level methods
where the projection space is defined on a coarser mesh. This leads to
additional couplings of degrees of freedoms belonging to neighbouring mesh
cell. Hence, the matrix stencil increases. One way to circumvent this
drawback is the use of one-level LPS methods where approximation space and
projection space are defined on the same mesh. In general, the
approximation spaces are enriched compared to standard spaces. However, the
additional degrees of freedom could be eliminated by static condensation.

A method using a grad-div stabilization to solve the transient Oseen
problem was studied in~\cite{FGJN16} where optimal estimates for the
divergence and the pressure were shown.

Time-dependent Oseen equations have been investigated in~\cite{Lube_Oseen}.
The local projection stabilization principle was applied to the derivative
in streamline-direction and to the divergence constraint separately.
Provided that the mesh width has the same order like the square root of the
viscosity, error estimates were obtained in~\cite{Lube_Oseen}. To
circumvent this unrealistic condition, the velocity approximation space
and the projection space have to fulfil local compatibility conditions.
The transient Navier--Stokes equations with this type of local projection
stabilization was investigated in~\cite{ADL15} where estimates for the
velocity error in the semi-discrete case were shown. The analysis of
the fully discretized problems using high-order term-by-term LPS methods
was considered in~\cite{ACJR16}.

Common time discretizations for incompressible flow problems are
$\vartheta$-schemes. Unfortunately, they are at most of second order like
the trapezoidal rule or the fractional-step $\vartheta$-scheme. In
addition, these methods do not provide a built-in mechanism for adaptive
time-step control. Just a few authors have considered higher order methods
in time like diagonally implicit Runge--Kutta methods, Rosenbrock--Wanner
methods, or continuous Galerkin--Petrov schemes of second order, see,
e.~g., \cite{JMR06,JR10,HST13}. To the best of our knowledge, no numerical
analysis has been provided for the first two classes applied to
incompressible flow problems or just convection-diffusion equations so far.
The cGP method in time and spatial stabilization by LPS method for the
transient Oseen equations was investigated in \cite{AJMN18}, where optimal
error bounds for velocity and pressure with constants that do not depends
on the viscosity parameter were obtained. Stepping schemes applied to the
time-dependent Navier--Stokes equations are given in~\cite{AM17}.

To discretise in time, we apply discontinuous Galerkin (dG) methods.
Discontinuous Galerkin methods were introduced for the first time 
in~\cite{RH73} to handle neutron transport problems. The analysis of dG
methods starts with~\cite{LR74}. For scalar hyperbolic problems,
theoretical investigation are given in~\cite{JP86}. Space-time dG
methods for convection-diffusion-reaction problems are studied
in~\cite{FHS07}. The dG methology was transferred to elliptic
problems by~\cite{WH1978} and to compressible and incompressible flow
problems, see~\cite{DF04, PV08} and the references therein. Temporal
discretizations of systems of ordinary differential equations by dG schemes
were introduced and analysed in~\cite{EEHJ96}. We also refer
to~\cite{Tho06}. Space-time dG finite element methods have been
applied to time-dependent advection-diffusion problems~\cite{SVD06}
and flow problems~\cite{VS08}. The combination of dG in time and LPS in space
for transient convection-diffusion-reaction equations has been studied
in~\cite{AMTX11}. Temporal discretizations using dG combined with an
equal-order interpolation for velocity and pressure applied to the
transient Stokes problem was studied in~\cite{ABM17}. Error estimates for
the semi-discrete and the fully discrete cases were proved. In~\cite{AJ15},
the dG and cGP methods in time combined with LPS and SUPG in space were considered for the convection-diffusion-reaction equation. It was shown that adaptive step control 
based on the post processed solution leads to the time step lengths that properly reflects
the dynamics of the solution. 

This paper studies the combination of the LPS method in space with
the dG($k$) method in time applied to the Navier--Stokes equations. For the
semi-discrete case, a stability results and an optimal estimate for
the velocity error will be given where all constants are independent of the
viscosity parameter $\nu$ in such sense that they depend just on higher
Sobolev norms of the solution $(\bu,p)$. Stability and convergence results
for the velocity error in the fully discrete case are proved. Also here,
the constants inside the estimates depend only via solution norms on the
viscosity. Moreover, an error estimate for the pressure in the fully
discrete case will be given. Unfortunately, the error constant depends on
the inverse of the smallest time step length.

The remainder of the paper is organised as follows:
Section~\ref{sec:prelim} considers preliminaries and provides used
notation. The error analysis for the semi-discrete problem are derived in
Sect.~\ref{sec:semi}. The temporal discretization by dG methods is given
in Sect.~\ref{sec:time_disc} where also stability properties are studied.
Moreover, error estimates for both velocity and pressure are showed.
Numerical results will be presented in Sect.~\ref{sec:numerics}.

\section{Preliminaries and notation}
\label{sec:prelim}
Throughout this paper, standard notation and conventions will be used.
For a measurable set $G\subset\mathbb{R}^d$, the inner product in $L^2(G)$
will be denoted by $(\cdot,\cdot)_G$. The norm and semi-norm in $W^{m,p}(G)$
are given by $\|\cdot\|_{m,p,G}$ and $|\cdot|_{m,p,G}$, respectively.
In the case $p=2$, we write $H^m(G)$, $\|\cdot\|_{m,G}$, and $|\cdot|_{m,G}$ 
instead of $W^{m,2}(G)$, $\|\cdot\|_{m,2,G}$, and $|\cdot|_{m,2,G}$.
If $G=\Omega$, the index $G$ in inner products, norms, and semi-norms will
be omitted.
Note that all definitions are extended to the cases of
vector-valued and tensor-valued arguments.
The subspace of functions from $H^1(\Omega)$
having zero boundary trace is denoted by $H^1_0(\Omega)$. 
The duality pairing between a
space $W$ and its dual $W'$ will be denoted by
$\langle \cdot,\cdot\rangle$. First and $j$-th order temporal derivatives
of a function $v$ are denoted by $v'$ and $v^{(j)}$, respectively.
Based on a Banach space $W$ with norm $\|\cdot\|_W$, the spaces
\begin{align*}
L^2(0,t;W) & \coloneqq  \left\{
v:[0,t]\to W\::\: \int_0^t \|v(s)\|_W^2~ds < \infty
\right\},\\
H^m(0,t;W) & \coloneqq  \left\{
v\in L^2(0,t;W) \::\: v^{(j)}\in L^2(0,t;W),\;1\le j\le m \right\},
\quad m\ge 1,\\
C(0,t;W) & \coloneqq  \left\{v:[0,t]\to W\::\: v\textrm{ is continuous
	with respect to time} \right\},\\
C^m(0,t;W) & \coloneqq \left\{v:[0,t]\to W\::\: v\textrm{ is $m$-times
	continuously differentiable in time}\right\}
\end{align*}
are defined
where the temporal derivatives $v^{(j)}$, $1\le j\le m$, in the definition
of $H^m(0,t;W)$ have to be understood in the sense of distributions.
We use in the case $t=T$ the abbreviations $L^2(W)$, $H^m(W)$, $C(W)$, and
$C^m(W)$ for the above defined spaces which are equipped with
\begin{alignat*}{2}
\|v\|_{L^2(W)} & \coloneqq  \left(\int_0^T\|v(t)\|^2_W\, dt\right)^{1/2},
&\qquad
\|v\|_{H^{m}(W)} & \coloneqq  \left(\sum_{j=0}^m
\|v^{(j)}\|_{L^2(W)}^2\right)^{1/2},
\\
\|v\|_{C(W)} & \coloneqq  \sup_{t\in I} \|v(t)\|_W,
&\qquad
\|v\|_{C^m(W)} & \coloneqq \sum_{j=0}^m \|v^{(j)}\|_{C(W)}
\end{alignat*}
as norms. In addition, we introduce
\[
L^\infty(W) := \left\{ v:[0,T]\to W \::\: \esssup_{0\le t\le T} \|v(t)\|_W 
< \infty \right\}
\]
where $\esssup$ denoted the essential supremum.

In order to derive a variational form of~\eqref{eq:nse}, we introduce the
spaces
\begin{equation*}
Q \coloneqq L^2_0(\Omega) = \big\{ q\in L^2(\Omega)\::\: (q,1) =
0\big\},\qquad \vv \coloneqq H^1_0(\Omega)^d.
\end{equation*}
Furthermore, let $C_F$ denote the Friedrichs constant
fulfilling
\begin{equation}
\label{eq:CF}
\|\bv\|_{\vv} \le C_F \|\nabla\bv\|_0\qquad\forall \bv\in\vv.
\end{equation}
In addition, we define
\[
\xx \coloneqq \big\{\bv\in L^2(\vv),\; \bv' \in L^2(\vv')\big\}
\]
where $\vv'=H^{-1}(\Omega)^d$ denotes the dual space of $\vv$. Note that
$\bv(0)$ is well-defined for $\bv\in\xx$ since the mapping
$\bv:[0,T]\to L^2(\Omega)^d$ is continuous.

A variational formulation of problem \eqref{eq:nse} reads: \medskip 

Find $\bu\in\xx$ with $\bu(0)=\bu_0$ and $p\in L^2(Q)$ such that
\begin{equation}\label{eq:weak_form}
\big\langle\bu'(t),\bv\big\rangle
+ A\left(  \big(\bu(t),p(t) \big), \big(\bv, q \big)\right)
+ \big((\bu(t)\cdot\nabla)\bu(t),\bv \big)
= \big\langle\ff(t), \bv\big\rangle \quad
\forall \bv\in \vv,\, q\in Q
\end{equation}
for almost all $t\in I$ where the bilinear form $A$ is given by
\begin{equation*}
A\big(  (\bv,q), (\w, r)\big)
\coloneqq  \nu(\nabla \bv, \nabla \w) - (q,\Div \w) + (r,\Div \bv).
\end{equation*}
Note that the initial condition $\bu(0)=\bu_0$ is well-defined since
$\bu\in\xx$.
For studying the existence of a velocity solution of~\eqref{eq:weak_form},
this system is usually considered in the subspace
\begin{equation*}
\vv^{\textrm{div}} \coloneqq  \big\{\bv \in \vv : (\Div \bv, q) = 0 \text{ for all } q\in
Q\big\}
\end{equation*}
of divergence-free functions. Using
\[
\xx^{\textrm{div}} \coloneqq
\big\{
\bv\in\xx\::\: \bv\in L^2\big(\vv^{\textrm{div}}\big)
\big\},
\]
the velocity solution
of~\eqref{eq:weak_form} can be computed by solving the problem:\medskip

Find $\bu \in \xx^{\textrm{div}}$ with $\bu(0)=\bu_0$ such that
\begin{equation}\label{eq:weak_form_pf}
\big\langle\bu'(t),\bv\big\rangle+\nu \big(\nabla \bu(t),
\nabla \bv\big)  + 
\big((\bu(t)\cdot\nabla)\bu(t),\bv \big)
= \big\langle\ff(t), \bv\big\rangle\quad
\forall \bv\in \vv^{\mathrm{div}}
\end{equation}
for almost all $t\in I$.

For the finite element discretization of \eqref{eq:weak_form}, let
$\{\mathcal T_h\}$ be a family of shape-regular triangulations of $\Omega$
into open $d$-simplices, quadrilaterals, or hexahedra.
The diameter of a mesh cell $K\in\mathcal{T}_h$ will be denoted by $h_K$
while the mesh size $h$ is defined as
$h\coloneqq \max\limits_{K\in\mathcal{T}_h}h_K$. For a mesh-cell
dependent quantity $\psi_K$, we will write $\psi_K\sim h_K^{\alpha}$ if
there are positive constants $A$ and $B$, both independent of $K$, such
that $A\,h_K^{\alpha}\le \psi_K\le B\,h_K^{\alpha}$ for all
$K\in\mathcal{T}_h$ and all $h$.

Let $Y_h \subset H^1_0 (\Omega)$
be a scalar finite element space of continuous, mapped piecewise
polynomial functions over $\mathcal{T}_h$. The finite element space
$\vv_h$ for approximating the velocity field is given by
$\vv_h \coloneqq  Y_h^d\cap \vv$. To discretise the pressure, let
$M_h\subset L^2(\Omega)$ denote a finite element space of continuous or
discontinuous functions with respect to $\mathcal{T}_h$. Furthermore, we
set $Q_h\coloneqq M_h\cap Q$. This paper considers inf-sup stable
pairs $(\vv_h,Q_h)$, i.e., there exists a positive constant $\beta_0$,
independent of $h$, such that
\begin{equation}\label{eq:disc_inf_sup}
\inf_{q_h\in Q_h\setminus\{0\}}
\sup_{\bv_h\in \vv_h\setminus\{\boldsymbol{0}\}} 
\frac{(\Div \bv_h, q_h)}{|\bv_h|_1 \|q_h\|_0}\ge \beta_0 > 0.
\end{equation}
Furthermore, we introduce
\begin{equation}
\label{eq:discdivfree}
\vv_h^{\mathrm{div}} \coloneqq  \big\{\bv_h \in \vv_h : (\Div \bv_h, q_h) = 0
\text{ for all } q_h\in M_h\big\}
\end{equation}
as space of discretely divergence-free functions. Note that
$\vv_h^{\mathrm{div}}$ can be equivalently defined using test functions
$q_h$ from $Q_h$ only since $\vv_h$ provides homogeneous Dirichlet boundary
conditions.

The semi-discrete standard Galerkin finite element method applied 
to~\eqref{eq:weak_form} reads
\medskip

Find $\bu_h\in H^1(\vv_h)$ with $\bu_h(0) = \bu_{0,h}$ and $p_h\in L^2(Q_h)$
such that
\begin{multline}\label{eq:galerkin}
\big( \bu_h'(t),\bv_h \big)+ A\left(  \big(\bu_h(t),p_h(t)\big),
\big(\bv_h, q_h\big)\right) +
n\big(\bu_h(t),\bu_h(t),\bv_h\big)
= \big\langle \ff(t), \bv_h\big\rangle\quad
\forall \bv_h\in \vv_h,\, q_h\in Q_h
\end{multline}
for almost all $t\in I$.
Note that
$\bu_{h,0}\in \vv_h$ is a suitable approximation of the initial velocity
$\bu_0$ in the finite element space $\vv_h$. Moreover, the initial
condition $\bu_h(0)=\bu_{h,0}$ is well-defined since $\bu_h\in H^1(\vv_h)$.
Furthermore, let $n$ denote the skew-symmetric form of the convective term
defined by
\begin{equation}
\label{skew}
n(\bu,\bv,\w) = \frac{1}{2} \Big[\big((\bu\cdot\nabla)\bv,\w \big)
- \big((\bu\cdot\nabla)\w,\bv\big)\Big].
\end{equation}
The trilinear form $n$ provides
\begin{equation}\label{eq:skewprop}
n(\bv,\w,\w) = 0 \qquad \forall \bv,\w\in \vv
\end{equation}
and
\begin{equation}
\label{eq:partint}
n(\bu,\bv,\w) = \big((\bu\cdot\nabla)\bv,\w\big)
+ \frac{1}{2} (\Div\bu,\bv\cdot\w)
\qquad\forall\bu,\bv,\w\in\vv.
\end{equation}
It is well-known that the standard Galerkin method~\eqref{eq:galerkin} is
unstable in the case of dominating convection unless $h$ is
unpractically small. The use of a stabilised discretization becomes
necessary.

This paper concentrates on the one-level variant of the local projection
stabilization method where
approximation space and projection space are defined on the same mesh.
For any $K\in \mathcal{T}_h$, let $D(K)$ be a finite-dimensional space and
$\pi_K : L^2(K)\to D(K)$ the associated local $L^2$-projection into
$D(K)$. The local fluctuation operator $\kappa_K : L^2(K)\to L^2(K)$
is given by $\kappa_K v\coloneqq v-\pi_K v$ and applied
component-wise to vector-valued and tensor-valued arguments. We define
\[
\kappa_h:L^2(\Omega)\to\bigoplus_{K\in\mathcal{T}_h} D(K),\qquad
\big(\kappa_h v\big)|_K \coloneqq \kappa_K\big(v|_K\big)\quad
\forall K\in\mathcal{T}_h
\]
as abbreviation.
Note that the estimate
\begin{equation}
\|\kappa_K\Div\bv\|_{0,K} \le \sqrt{d}\|\kappa_K\nabla\bv\|_{0,K},
\qquad K\in\mathcal{T}_h,\,\bv\in\vv,
\label{eq:kappadiv}
\end{equation}
holds true.

The stabilization term $S_h$ is defined by
\begin{align*}
S_h(\bv_h,\w_h) \coloneqq  \sum_{K\in \mathcal{T}_h}
\mu_K\big(\kappa_K \nabla \bv_h, \kappa_K\nabla \w_h\big)_K
\end{align*}
where $\mu_K$, $K\in\mathcal{T}_h$, are user-chosen non-negative constants.
Furthermore, we set
\[
\mu_h^{\textrm{min}} \coloneqq \min_{K\in\mathcal{T}_h} \mu_K,
\qquad
\mu_h^{\textrm{max}} \coloneqq \max_{K\in\mathcal{T}_h} \mu_K.
\]
The precise choice of $\mu_K$ will be discussed in the upcoming sections.
Note that also the separate stabilization of the divergence constraint and
the derivative in streamline direction is possible,
see~\cite{MST07,MT15,ADL15}.

The stabilised semi-discrete problem reads:\medskip

Find $\bu_h\in H^1(\vv_h)$ with $\bu_h(0) = \bu_{0,h}$ and $p_h\in L^2(Q_h)$
such that
\begin{multline}\label{eq:lps}
\big( \bu_h'(t),\bv_h \big) + A_h \left( \big(\bu_h(t),p_h(t)\big),
\big(\bv_h, q_h\big)\right) +
n\big(\bu_h(t),\bu_h(t),\bv_h\big)\\
= \big\langle \ff(t), \bv_h\big\rangle\quad
\forall \bv_h\in \vv_h,\, q_h\in Q_h
\end{multline}
for almost all $t\in I$ where $A_h$ is given by 
\begin{equation*}
A_h\left(  \big(\bv,q \big), \big(\w, r \big)\right)
=  A\left(  \big(\bv,p \big), \big(\w, r \big)\right)  + S_h(\bv,\w).
\end{equation*}
Note that
\begin{equation}
\label{eq:Ah_coer}
A_h\big((\bv,q),(\bv,q)\big) = \nu\|\nabla\bv\|_0^2 + S_h(\bv,\bv)
\end{equation}
for all $(\bv,q)\in \vv\times Q$.

For our subsequent analysis, several assumptions on $\vv_h$, $M_h$, and
$D(K)$ will be made. Note that $r\ge 2$ will be a fixed integer describing
the order of the spatial discretization. The dependence of constants on $r$
will not be elaborated in this paper.
\begin{assumption}\label{assmption_a1}
	There exists an interpolation operator $j_h:\vv\to\vv_h$ which
	provides for $p\in[2,\infty]$ the approximation property
	\begin{equation}
	\label{j1}
	\|\w - j_h\w\|_{0,p} + h |\w - j_h\w|_{1,p} \le C h^{\ell} \|\w\|_{\ell,p}
	\qquad\forall \w\in W^{\ell,p}(\Omega)^d,\,1\le\ell\le r+1,
	\end{equation}
	and preserves the discrete divergence
	\begin{equation}
	\label{discdiv}
	\big(\Div(\w - j_h\w),q_h\big) = 0\qquad\forall q_h\in
	Q_h,\,\w\in\vv.
	\end{equation}
	In addition, there is an interpolation operator $i_h:L^2(\Omega)\to M_h$ with
	$i_h q\in Q_h$ for $q\in Q$ which guarantees the approximation property
	\begin{alignat}{2}
	\label{j2}
	\|q - i_h q\|_{0,p} + h |q - i_h q|_{1,p} & \le C h^{\ell} \|q\|_{\ell,p}
	&\qquad&\forall q\in W^{\ell,p}(\Omega),\,1\le\ell\le r,\\
	\intertext{the stability}
	\label{js}
	\|q - i_h q\|_0 & \le C_i \|q\|_0 &\qquad&\forall q\in L^2(\Omega),
	\intertext{and}
	\label{j3}
	(q - i_h q, r_h)_K & = 0 &\qquad&\forall q\in L^2(\Omega),\,r_h\in D(K)
	\end{alignat}
	for all $K\in\mathcal{T}_h$.
\end{assumption}

The existence of velocity interpolation operators $j_h$
fulfilling~\eqref{j1} and~\eqref{discdiv} has been studied
in~\cite{GS03}.
In the case of discontinuous pressure approximations with $D(K)\subset
M_h|_K$, the $L^2(\Omega)$-projection into $M_h$ fulfils~\eqref{j2},
\eqref{js}, and~\eqref{j3} since it localises to the $L^2(K)$-projections.
A detailed discussion on interpolation operators satisfying~\eqref{j3} can
be found in~\cite{MT15}. Note that~\eqref{j1} ensures the bounds
\begin{equation}
\label{jstab}
\|\nabla j_h \w\|_{0,p} \le C \|\w\|_{1,p},\quad
\|j_h \w\|_{0,p} \le C \|\w\|_{1,p},
\qquad p\in[2,\infty],\, \w\in W^{1,p}(\Omega)^2,
\end{equation}
hence, the interpolation operator $j_h$ is stable.

\begin{assumption}\label{assmption_a2}
	The fluctuation operator provides the approximation property
	\begin{equation}\label{kappa}
	\big\|\kappa_K q\big\|_{0,K} \leq C h^l_K\big|q\big|_{l,K}
	\qquad \forall   K\in \mathcal {T}_h,\;
	\forall   q\in H^l(K), \; 0\leq l \leq r.
	\end{equation}
\end{assumption}
Projection spaces $D(K)$ which guarantee~\eqref{kappa} are given
in~\cite{MST07}.

Finally, we mention that the combination $\vv_h=Q_r$,
$M_h=P_{r-1}^{\text{disc}}$ with $D(K)=P_{r-1}(K)$ fulfils for $r\ge 2$
on quadrilateral/hexahedral meshes all assumptions. For details,
we refer to~\cite{GS03,MST07,MT15}.

\section{Error analysis for the semi-discrete case}
\label{sec:semi}
This section considers stability properties and error estimates for the
stabilised semi-discrete problem~\eqref{eq:lps}.

The following lemma states the stability of the velocity solution $\bu_h$.
\begin{lemma}
	Let $\bu_{0,h}\in \vv_h$ and $ \ff\in L^2(\vv')$. Then
	problem~\eqref{eq:lps} satisfies the stability estimate
	\begin{equation}\label{eq:sd_stability}
	\|\bu_h(t)\|_0^2 + \int_0^t\nu \|\nabla \bu_h(s)\|_0^2
	+ 2 S_h\big( \bu_h(s),\bu_h(s)\big)\,ds
	\le \|\bu_{0,h}\|_0^2 + \frac{C_F^2}{\nu} \int_0^t \|\ff(s)\|_{\vv'}^2\,ds
	\end{equation}
	for almost all $t\in (0,T)$ where $C_F$ is the Friedrichs
	constant from~\eqref{eq:CF}.
\end{lemma}
\begin{proof}
	The statement follows by setting $(\bv_h,q_h) = \big(\bu_h(t),p_h(t)\big)$
	in~\eqref{eq:lps}, using the skew-symmetry~\eqref{eq:skewprop} of the
	trilinear form $n$, the coercivity property \eqref{eq:Ah_coer} of $A_h$,
	the properties of the duality pairing between $\vv$ and $\vv'$, the
	Friedrichs inequality~\eqref{eq:CF}, an integration over the time interval
	$(0,t)$, and Young's inequality applied to the right-hand side.
\end{proof}

Provided $\ff$ is more regular, a $\nu$-independent bound can be shown.
\begin{lemma}
	Assuming the regularity $\ff \in L^1(L^2)$, the
	bound
	\begin{equation}\label{eq:sd_stability_improved}
	\frac{1}{2}  \|\bu_h(t)\|_0^2 + \int_0^t\nu \|\nabla \bu_h(s)\|_0^2
	+  S_h\big( \bu_h(s),\bu_h(s)\big)\,ds
	\le C_\mathrm{S}
	\end{equation}
	is obtained for almost all $t\in (0,T)$ where
	\begin{equation}
	\label{eq:CS}
	C_\mathrm{S}\coloneqq\|\bu_{0,h}\|_0^2 + \frac{3}{2}  \|\ff\|_{L^1(L^2)}^2
	\end{equation}
	is a constant depending on the problem data only.
\end{lemma}
\begin{proof}
	This result is a straightforward adaption of Lemma~3.1
	by~\cite{ADL15} where
	a local projection scheme with separate stabilization of streamline
	derivative and divergence constraint was considered.
\end{proof}

Our analysis will exploit that some appearing functions belong to the space
\begin{equation}
\label{eq:vhs}
\widetilde{\vv}_h^{\mathrm{div}} \coloneqq
\big\{ \bv\in\vv\::\: (\Div\bv,q_h) = 0 \text{ for all } q_h\in M_h \big\}
\end{equation}
that covers $\vv^{\mathrm{div}}+\vv_h^{\mathrm{div}}$. We frequently use
following estimate.
\begin{lemma}\label{lem:discdiv}
	Let $\bv_h\in\widetilde{\vv}_h^{\textrm{div}}$.
	Then, the estimate
	\begin{equation*}
	\big|(\Div\bv_h,\boldsymbol{\varphi}\cdot\boldsymbol{\psi})\big|
	\le C_d S_h(\bv_h,\bv_h)^{1/2}
	\|\boldsymbol{\varphi}\|_{0,\infty} \|\boldsymbol{\psi}\|_0,\qquad
	\boldsymbol{\varphi}\in L^{\infty}(\Omega)^d,\; \boldsymbol{\psi}\in L^2(\Omega)^d,
	\end{equation*}
	holds true where
	\begin{equation}
	\label{eq:Cd}
	C_d\coloneqq \frac{C_i\,\sqrt{d}}{\sqrt{\mu_h^{\textrm{min}}}}
	\end{equation}
	with $C_i$ from~\eqref{js}.
\end{lemma}
\begin{proof}
	We obtain
	\begin{align*}
	(\Div\bv_h,\boldsymbol{\varphi}\cdot\boldsymbol{\psi}) 
	= \big(\Div\bv_h,\boldsymbol{\varphi}\cdot\boldsymbol{\psi}
	- i_h(\boldsymbol{\varphi}\cdot\boldsymbol{\psi})\big)
	= \sum_{K\in\mathcal{T}_h} \big(\Div\bv_h - \pi_K\Div\bv_h,
	\boldsymbol{\varphi}\cdot\boldsymbol{\psi} - i_h(\boldsymbol{\varphi}\cdot\boldsymbol{\psi})\big)_K
	\end{align*}
	using the properties of $\widetilde{\vv}_h$ and $i_h$,
	in particular~\eqref{j3}. A generalised H\"older inequality
	and~\eqref{eq:kappadiv} yield the statement of this lemma.
\end{proof}

An error estimate for the velocity is given in the next theorem.

\begin{theorem}\label{thm:semi_disc_vel_estimates}
	Let the finite element spaces $\vv_h$ and $Q_h$ satisfy the discrete
	inf-sup condition~\eqref{eq:disc_inf_sup}. Suppose
	assumptions~\ref{assmption_a1}, \ref{assmption_a2}, and $\mu_K\sim 1$
	for all $K\in  \mathcal{T}_h$. Let $(\bu,p)$ be the solution of the
	continuous problem~\eqref{eq:weak_form} and $(\bu_h,p_h)$ be the solution
	of the stabilised semi-discrete problem~\eqref{eq:lps} with the initial
	condition $\bu_{0,h} = j_h\bu_0$. In addition, we assume
	\[
	\bu\in L^2(W^{r+1,\infty}),\quad
	\bu\in L^{\infty}(W^{r,\infty}),\quad
	\bu'\in L^2(H^{r+1}),\quad
	p\in L^2(H^r).
	\]
	Then, the error estimate
	\begin{equation}\label{est:velocity}
	\esssup_{0\le t \le T}\|(\bu-\bu_h)(t)\|_0^2
	+ \int_0^T \big[ \nu \|\nabla (\bu-\bu_h)\|_0^2
	+ S_h(\bu-\bu_h,\bu-\bu_h) \big]
	\le C_{\mathrm{exp}} C(\bu) h^{2r}
	\end{equation}
	holds true where
	\begin{equation}
	\label{eq:cexp}
	C_{\mathrm{exp}} \coloneqq \exp\left(
	2 \|\nabla\bu\|_{L^1(L^{\infty})} + C_d^2 \|\bu\|_{L^2(L^{\infty})}^2
	+ 3 T
	\right)
	\end{equation}
	is a constant depending on $\bu$ and $T$ while
	\begin{align}
	\label{eq:cu}
	C(\bu) \coloneqq C\bigg(
	& \|\bu'\|_{L^2(H^r)}^2
	+ \|\bu\|_{L^\infty(W^{r,\infty})}^2
	+ \|\bu\|_{L^2(W^{r+1,\infty})}
	+ \|p\|_{L^2(H^r)}^2
	\\
	& + \|\bu\|_{L^2(H^r)}^2 \|\bu\|_{L^{\infty}(W^{1,\infty})}^2
	+ \|\bu\|_{L^2(H^{r+1})}^2 \|\bu\|_{L^{\infty}(L^{\infty})}^2
	\nonumber
	\bigg)
	\end{align}
	depends on $\bu$, $p$, $C_{\mathrm{S}}$, $C_d$, $\mu_h^{\textrm{max}}$,
	$\mu_h^{\textrm{min}}$, $d$, and the interpolation constants of $i_h$ and
	$j_h$.
\end{theorem}

\begin{proof}
	Defining $\be \coloneqq  \bu-j_h \bu$ and $\bx_h \coloneqq j_h \bu -
	\bu_h$, the error splitting $\bu -\bu_h = \be + \bx_h$ holds
	for the velocity. After subtracting~\eqref{eq:lps} from~\eqref{eq:weak_form},
	a straightforward  calculation leads for
	$\bx_h\in H^1\big(\vv_{h}^{\Div}\big)$ and $\bv_h\in
	L^2\big(\vv_{h}^{\textrm{div}}\big)$
	to the error equation
	\begin{align}\label{eq:error_equation_sdp}
	(\bx_h', \bv_h) + \nu(\nabla \bx_h,\nabla \bv_h) + S_h(\bx_h,\bv_h)
	=& S_h(j_h\bu, \bv_h) - 
	(\be', \bv_h) - \nu(\nabla \be,\nabla \bv_h)\nonumber\\
	& - n(\bu,\bu,\bv_h)+ n(\bu_h,\bu_h,\bv_h)
	+ (p - i_h p,\Div \bv_h)	
	\end{align}
	where we skip writing the time dependence. With $\bv_h=\bx_h$, the
	Cauchy--Schwarz inequality and Young's inequality give
	\begin{align*}
	\frac{1}{2}\frac{d}{dt} \|\bx_h\|_0^2 + \nu \|\nabla \bx_h\|_0^2
	+ S_h(\bx_h,\bx_h)
	& \le S_h(j_h\bu,j_h\bu) + \frac{1}{4} S_h(\bx_h,\bx_h)
	+ \frac{1}{2} \|\be'\|_0^2 + \frac{1}{2} \|\bx_h\|_0^2\\
	&\qquad + \frac{\nu}{2} \|\nabla\be\|_0^2 + \frac{\nu}{2} \|\nabla\bx_h\|_0^2
	+ \frac{2d}{\mu_h^{\textrm{min}}} \|p-i_h p\|_0^2
	\\
	& \qquad
	+ \frac{1}{8} S_h(\bx_h,\bx_h)
	+\big|n(\bu,\bu,\bx_h)-n(\bu_h,\bu_h,\bx_h)\big|
	\end{align*}
	where we used
	\[
	(p - i_h p,\Div\bx_h) = \sum_{K\in\mathcal{T}_h} (p - i_h p,
	\kappa_K \Div\bx_h)_K \le \frac{\sqrt{d}}{\sqrt{\mu_h^{\textrm{min}}}}
	\|p - i_h p\|_0 \, S_h(\bx_h,\bx_h)^{1/2}
	\]
	due to property~\eqref{j3} of $i_h$ and~\eqref{eq:kappadiv}.
	The difference of the nonlinear terms is decomposed as
	\begin{align*}
	n(\bu,\bu,\bx_h)-n(\bu_h,\bu_h,\bx_h) 
	&= n(\bu-\bu_h,\bu,\bx_h) + n(\bu_h,\bu-\bu_h,\bx_h)\\
	&= n(\be,\bu,\bx_h) + n(\bx_h,\bu,\bx_h)
	+ n(\bu_h,\be,\bx_h)
	\end{align*}
	where $n(\bu_h,\bx_h,\bx_h)=0$ due to~\eqref{eq:skewprop} was used.
	The three nonlinear terms are estimated separately using Young's
	inequality, the representation~\eqref{eq:partint} of the trilinear form
	$n$, and generalised H\"older inequalities. Hence, we have
	\begin{align*}
	\big|n(\be,\bu,\bx_h)\big|
	& \le \|\be\|_0 \|\nabla\bu\|_{L^{\infty}} \|\bx_h\|_0
	+ \frac{1}{2} \|\Div \be\|_0 \|\bu\|_{L^{\infty}} \|\bx_h\|_0\\
	& \le \|\nabla\bu\|_{L^{\infty}}^2 \|\be\|_0^2
	+ \frac{1}{4} \|\bu\|_{L^{\infty}}^2 \|\Div\be\|_0^2
	+ \frac{1}{2} \|\bx_h\|_0^2.
	\end{align*}
	Using Lemma~\ref{lem:discdiv} and Young's inequality, we obtain
	\begin{align*}
	\big|(\Div\bx_h,\bu\cdot\bx_h)\big|
	\le \frac{1}{4}S_h(\bx_h,\bx_h) + C_d^2
	\|\bu\|_{L^{\infty}}^2 \|\bx_h\|_0^2
	\end{align*}
	with $C_d$ from~\eqref{eq:Cd}.
	This results in
	\begin{align*}
	\big|n(\bx_h,\bu,\bx_h)\big| & \le \|\bx_h\|_0^2
	\|\nabla\bu\|_{L^{\infty}}
	+ \frac{1}{2} \big|(\Div\bx_h,\bu\cdot\bx_h)\big|\le \frac{1}{8} S_h(\bx_h,\bx_h) + 
	\left(\frac{C_d^2}{2}\|\bu\|_{L^{\infty}}^2 + \|\nabla\bu\|_{L^{\infty}}
	\right)\|\bx_h\|_0^2.
	\end{align*}
	The third nonlinear term can be estimated as follows
	\begin{align*}
	\big|n(\bu_h,\be,\bx_h)\big|
	& \le \|\bu_h\|_0\|\nabla\be\|_{L^{\infty}}\|\bx_h\|_0
	+ \frac{C_d}{2}S_h(\bu_h,\bu_h)^{1/2} \; \|\be\|_{L^{\infty}} \|\bx_h\|_0\\
	& \le \|\bu_h\|_0^2 \|\nabla\be\|_{L^{\infty}}^2 
	+ \frac{C_d^2}{4}S_h(\bu_h,\bu_h) \|\be\|_{L^{\infty}}^2
	+ \frac{1}{2} \|\bx_h\|_0^2
	\end{align*}
	using Lemma~\ref{lem:discdiv} and Young's inequality since
	$\bu_h\in H^1(\vv_h^{\textrm{div}})$.
	Inserting these estimates and putting similar terms to the left-hand side,
	one gets
	\begin{multline}\label{eq:sd_vel_est_err_eq}
	\frac{1}{2}\frac{d}{dt} \|\bx_h\|_0^2 + \frac{\nu}{2} \|\nabla \bx_h\|_0^2 
	+ \frac{1}{2} S_h(\bx_h,\bx_h)\\
	\begin{aligned}
	&\le S_h(j_h\bu,j_h\bu)+ \frac{1}{2}\|\be'\|_0^2
	+ \frac{\nu}{2} \|\nabla \be\|_0^2
	+ \frac{2d}{\mu_h^{\textrm{min}}}\|p-i_hp\|_0^2\\
	&\qquad + \|\nabla\bu\|_{L^{\infty}}^2 \|\be\|_0^2
	+ \frac{1}{4} \|\bu\|_{L^{\infty}}^2 \|\Div\be\|_0^2
	+ \|\bu_h\|_0^2 \|\nabla\be\|_{L^{\infty}}^2
	\\
	&\qquad
	+ \frac{C_d^2}{4} S_h(\bu_h,\bu_h) \|\be\|_{L^{\infty}}^2
	+ \left(
	\|\nabla\bu\|_{L^{\infty}} + \frac{C_d^2}{2}
	\|\bu\|_{L^{\infty}}^2 +\frac{3}{2} \right) \|\bx_h\|_0^2.
	\end{aligned}
	\end{multline}
	Multiplying by $2$ and integrating the above estimate over $(0,t)$ lead to
	\begin{multline}\label{eq:prf_sd_1}
	\|\bx_h(t)\|_0^2 + \int_0^t \big[\nu \|\nabla \bx_h\|_0^2
	+ S_h(\bx_h,\bx_h)\big]\\
	\begin{aligned}
	& \le  \int_0^t \bigg[ 2 S_h(j_h \bu, j_h \bu)
	+  \|\be'\|_0^2 + \nu \|\nabla \be\|_0^2
	+ \frac{4d}{\mu_h^{\textrm{min}}}\|p-i_hp\|_0^2 \bigg]\\
	& \qquad + \int_0^t \left[
	2\|\nabla\bu\|_{L^{\infty}}^2\|\be\|_0^2
	+ \frac{1}{2} \|\bu\|_{L^{\infty}}^2 \|\Div\be\|_0^2
	\right]\\
	& \qquad + \int_0^t \left[
	2 \|\bu_h\|_0^2 \;\|\nabla \be\|_{L^\infty}^2
	+ \frac{C_d^2}{2} S_h(\bu_h,\bu_h) \|\be\|_{L^{\infty}}^2
	\right]\\
	& \qquad
	+ \int_0^t \left(
	2 \|\nabla\bu\|_{L^{\infty}}
	+ C_d^2 \|\bu\|_{L^{\infty}}^2
	+3 \right) \|\bx_h\|_0^2
	\end{aligned}
	\end{multline}
	where we used $\bx_h(0)=\boldsymbol{0}$ due to the choice $\bu_{0,h} = j_h\bu_0$ of
	the discrete initial condition. Using the $L^2$-stability of the
	fluctuation operator $\kappa_K$, the properties~\eqref{j1}
	of the interpolation operator $j_h$, and the approximation
	property~\eqref{kappa} of $\kappa_K$, we get for the first term on the
	right-hand side of~\eqref{eq:prf_sd_1}
	\begin{equation*}
	\int_0^t S_h(j_h\bu,j_h\bu) 
	\le 2 \int_0^t \big[ S_h(j_h\bu - \bu, j_h\bu - \bu) + S_h(\bu,\bu)\big]
	\le C \,\mu_h^{\textrm{max}}\, h^{2r} \int_0^t \|\bu\|_{r+1}^2.
	\end{equation*}
	The stability estimate \eqref{eq:sd_stability_improved} provides
	\begin{equation*}
	\int_0^t\|\bu_h\|_0^2 \; \|\nabla \be\|_{L^\infty}^2 
	\le \|\bu_h\|_{L^\infty(L^2)}^2 \int_0^t \|\nabla \be\|_{L^\infty}^2 \le 
	2C_\textrm{S} \int_0^t \|\nabla \be\|_{L^\infty}^2
	\end{equation*}
	and
	\begin{equation*}
	\int_0^t S_h(\bu_h,\bu_h)\|\be\|_{L^{\infty}}^2
	\le \|\be\|_{L^{\infty}(L^{\infty})}^2 \int_0^t S_h(\bu_h,\bu_h)
	\le C_\textrm{S} \|\be\|_{L^{\infty}(L^{\infty})}^2.
	\end{equation*}
	In addition, we have
	\begin{align*}
	\int_0^t \|\nabla\bu\|_{L^{\infty}}^2\|\be\|_0^2
	& \le C h^{2r} \|\bu\|_{L^{\infty}(W^{1,\infty})}^2
	\|\bu\|_{L^2(H^r)}^2
	\intertext{and}
	\int_0^t \|\bu\|_{L^{\infty}}^2\|\Div\be\|_0^2
	& \le C h^{2r} \|\bu\|_{L^{\infty}(L^{\infty})}^2
	\|\bu\|_{L^2(H^{r+1})}^2.
	\end{align*}
	Using the above bounds and the approximation properties~\eqref{j1}
	and~\eqref{j2} of the interpolation operators $j_h$ and $i_h$
	in~\eqref{eq:prf_sd_1}, we get
	\begin{equation*}
	\|\bx_h(t)\|_0^2 + \int_0^t \big[ \nu \|\nabla \bx_h\|_0^2 
	+ S_h(\bx_h,\bx_h) \big]
	\le C(\bu) h^{2r} 
	+ \int_0^t \left(
	2 \|\nabla\bu\|_{L^{\infty}} + C_d^2
	\|\bu\|_{L^{\infty}}^2
	+3 \right) \|\bx_h\|_0^2
	\end{equation*}
	with $C(\bu)$ given in~\eqref{eq:cu}. Then, the application of
	Gronwall's lemma leads for almost all $t\in(0,T)$ to
	\begin{equation*}
	\|\bx_h(t)\|_0^2 + \int_0^t \big[\nu \|\nabla \bx_h\|_0^2
	+ S_h(\bx_h,\bx_h) \big] 
	\le C_{\mathrm{exp}} C(\bu) h^{2r}
	\end{equation*}
	with the Gronwall constant $C_{\mathrm{exp}}$ defined
	in~\eqref{eq:cexp}. This estimate, the application of the
	triangle inequality
	\begin{multline*}
	\esssup_{0\le t \le T} \|(\bu-\bu_h)(t)\|_0^2
	+ \int_0^T \big[ \nu \|\nabla (\bu-\bu_h)\|_0^2
	+ S_h(\bu-\bu_h,\bu-\bu_h) \big] \\
	\;\;\le 2\bigg\{ \esssup_{0\le t \le T}\|\be(t)\|_0^2
	+ \int_0^T \big[ \nu \|\nabla \be\|_0^2
	+  S_h(\be,\be) \big] 
	+ \esssup_{0\le t \le T} \|\bx_h(t)\|_0^2
	+ \int_0^T \big[ \nu \|\nabla \bx_h\|_0^2
	+ S_h(\bx_h,\bx_h) \big]
	\bigg\},
	\end{multline*}
	and the approximation properties~\eqref{j1} conclude the proof.
\end{proof}
\section{Time discretization by discontinuous Galerkin method}
\label{sec:time_disc}
We discretise in this section the semi-discrete problem~\eqref{eq:lps} in
time by using discontinuous Galerkin (dG) methods to obtain a fully
discrete LPS/dG formulation of~\eqref{eq:weak_form}. To this end, we consider a
partition $0=t_0<t_1<\cdots<t_N=T$ of the time interval $I=[0,T]$ and set
$I_n\coloneqq  (t_{n-1},t_n]$, $\tau_n\coloneqq t_n-t_{n-1}$, $n=1,\ldots,N$, and
\begin{equation}\label{taumin}
\tau \coloneqq  \max_{1\le n\le N}\tau_n \le 1,\qquad
\tau_{\textrm{min}} \coloneqq  \min_{1\le n\le N}\tau_n.
\end{equation}
For non-negative integers $k$, we
define the fully discrete time-discontinuous velocity and pressure spaces
as follows:
\begin{align*}
\xx_k & \coloneqq  \Big\{ \bv_h\in L^2(I,\vv_h):\bv_h|_{I_n}\in
\mathbb{P}_k (I_n,\vv_h),\;
n=1,\dots,N \Big\},\\
\xx_k^{\textrm{div}} & \coloneqq  \Big\{ \bv_h\in\xx_k
:\bv_h|_{I_n}\in \mathbb{P}_k
(I_n,\vv_h^{\textrm{div}}),\;
n=1,\dots,N \Big\},\\
Y_k & \coloneqq  \Big\{ q_h\in L^2(I,Q_h):q_h|_{I_n}\in
\mathbb{P}_k (I_n,Q_h),\;
n=1,\dots,N \Big\},
\end{align*}
where
\begin{align*}
\mathbb{P}_k(I_n,W_h) \coloneqq  \left\{ w_h :I_n\to W_h: 
w_h(t)=\sum_{j=0}^k W^jt^j,\;
W^j\in W_h, \; j=0,\dots,k \right\}
\end{align*}
denotes the space of $W_h$-valued polynomials of degree less than or
equal to $k$ in time.
For a function $w$ being piecewise smooth in time, we define at $t=t_n$
the left-sided value~$w_n^-$, the right-sided value~$w_n^+$, and the
jump~$[w]_n$ as
\[
w_n^- \coloneqq  \lim_{t\rightarrow t_n-0} w(t),\qquad
w_n^+ \coloneqq  \lim_{t\rightarrow t_n+0} w(t),\qquad
[w]_n \coloneqq w_n^+-w_n^-.
\] 
The discontinuous Galerkin method applied to~\eqref{eq:lps} leads to the 
fully discrete problem\medskip

Find $(\uht,\pht)\in \xx_k\times Y_k$ such that
\begin{multline}\label{eq:fdp_with_jump}
\sum_{n=1}^N \intin \big[(\uht', \vht)
+ A_h\big((\uht, \pht),(\vht, \qht)\big)
+ n(\uht,\uht,\vht)\big]\\
+ \sum_{n=1}^{N-1} \big([\uht]_n, \vht(t_n^+)\big)
+ (\uht(t_0^+), \vht(t_0^+))
= \big( j_h \bu_0,\vht(t_0^+)\big)
+ \int_0^T \big\langle\ff,\vht\big\rangle \qquad 
\end{multline}
\indent for all $\vht\in \xx_k$ and all $\qht\in Y_k$.
\medskip

\noindent Note that the initial condition is enforced only weakly.

In order to evaluate the time integrals in~\eqref{eq:fdp_with_jump}
numerically, the right-sided Gau\ss--Radau
quadrature with $(k+1)$ points will be applied. Let $-1<\hat{t}_1
<\dots < \hat{t}_{k+1}=1$ and $\widehat{\omega}_j$, $j=1,\dots,k+1$, denote the
points and weights of this quadrature formula on the reference time
interval $[-1,1]$. We define on $I_n$, $n=1,\dots, N$, the transformed
quadrature formula $Q_n$ by
\[
\Qn{\varphi} \coloneqq  \frac{\tau_n}{2} \sum_{j=1}^{k+1} \widehat{\omega}_j
\varphi(t_{n,j}),\qquad
t_{n,j}\coloneqq T_n(\hat{t}_j),
\]
where
\begin{equation}
\label{Tn}
T_n:[-1,1]\to \overline{I}_n,\quad \hat{t}\mapsto
t_{n-1}+\frac{\tau_n}{2}(\hat{t}+1),
\end{equation}
is an affine mapping, see~\cite{MS11}. Note that $Q_n$ integrates
polynomials of degree less than or equal to $2k$ exactly. Moreover, $Q_n$
fulfils a Cauchy--Schwarz-like estimate
\begin{equation}
\label{eq:QnCSU}
\Qn{\varphi\psi} \le \Qn{\varphi^2}^{1/2} \Qn{\psi^2}^{1/2}
\end{equation}
for all suitable functions $\varphi$ and $\psi$. Furthermore, we set
\[
\Q{\varphi}\coloneqq \sum_{n=1}^N \Qn{\varphi}
\]
as abbreviation.

Let us introduce the forms $B_n$ and $B_{h,n}$ as 
\begin{align}
B_n\big(\boldsymbol{\varphi};(\bv,q), (\w, r)\big) &\coloneqq 
\Qn{(\bv', \w) + A\big((\bv,q),(\w,r)\big) +
	n(\boldsymbol{\varphi},\bv,\w)}
+ \big(\bv(t_{n-1}^+), \w(t_{n-1}^+)\big)
\label{eq:B}
\intertext{and}
B_{h,n}\big(\boldsymbol{\varphi};(\bv,q), (\w, r)\big) &\coloneqq 
\Qn{(\bv', \w) + A_h\big((\bv,q),(\w,r)\big) +
	n(\boldsymbol{\varphi},\bv,\w)}+ \big(\bv(t_{n-1}^+), \w(t_{n-1}^+)\big).
\label{eq:Bh}
\end{align}
Note that $B_n$ and $B_{h,n}$ are linear with respect to their second and
third arguments.

If the solution $(\bu,p)$ of~\eqref{eq:weak_form} belongs to
$C^1(\vv)\times C(Q)$, we have
\begin{align}\label{eq:new_weak_form}
B_n\big(\bu;(\bu,p),(\bv,q)\big) = \big(\bu(t_{n-1}^-), \bv(t_{n-1}^+)\big)
+ \Qn{\big\langle\ff,\bv\big\rangle}\qquad
\forall (\bv,q)\in L^2(\vv)\times L^2(Q)
\end{align}
where $\bu(t_0^{-}) = \bu_0$.

Since the test functions $\vht \in \xx_k$ and $\qht\in Y_k$ are allowed to
be discontinuous at the discrete time points $t_n$, we can choose their
values on the time intervals $I_n$, $n=1,\dots,N$, independently. By
considering $\vht$ and $\qht$ to vanish outside the time interval $I_n$,
the fully discrete scheme~\eqref{eq:fdp_with_jump} results in a sequence of
local problems on each $I_n$, $1\le n\le N$, which read\medskip

For given $\bu(t_{n-1}^-)$, find $\uht\big|_{I_n}\in\mathbb{P}_k(I_n,\vv_h)$
and $\pht\big|_{I_n} \in \mathbb{P}_k(I_n,Q_h)$ such that 
\begin{equation}\label{eq:fdp_one_form}
B_{h,n}\big(\uht; (\uht,\pht),(\vht,\qht)\big) = 
\big(\uht(t_{n-1}^-), \vht(t_{n-1}^+)\big) +
\Qn{\big\langle\ff,\vht\big\rangle}
\end{equation}
for all $\vht\in \mathbb{P}_k(I_n,\vv_h)$ and $\qht \in
\mathbb{P}_k(I_n,Q_h)$ where $\uht(t^-_0) = j_h \bu_0$.

\subsection{Representation of the fully discrete scheme}
In order to get an algebraic formulation of~\eqref{eq:fdp_one_form}, let
$\widehat{\varphi}_1,\dots,\widehat{\varphi}_{k+1}\in\mathbb{P}_k$
denote the Lagrange basis functions
with respect to the Gau\ss--Radau points $\hat{t}_1,\dots,\hat{t}_{k+1}$ on
$[-1,1]$. Following~\cite{MS11}, we define
\begin{equation}
\label{phinj}
\varphi_{n,j}(t) \coloneqq  \widehat{\varphi}_j\big(T_n^{-1}(t)\big)
\end{equation}
on $I_n$, $n=1,\dots,N$, with $T_n$ from~\eqref{Tn}. Since the restrictions
of $\uht$ and $p_{h,\tau}$ to the interval $I_n$ are $\vv_h$-valued and
$Q_h$-valued polynomials of degree less than or equal to $k$, they can
be represented as
\begin{equation}
\label{fdsol}
\uht\big|_{I_n}(t) \coloneqq  \sum_{j=1}^{k+1} \boldsymbol{U}_{n,h}^j
\varphi_{n,j}(t),\qquad
p_{h,\tau}\big|_{I_n}(t) \coloneqq  \sum_{j=1}^{k+1} P_{n,h}^j \varphi_{n,j}(t)
\end{equation}
with $(\boldsymbol{U}_{n,h}^j, P_{n,h}^j)\in \vv_h\times Q_h$,
$j=1,\ldots,k+1$. 
The choice of the ansatz basis guarantees
\[
\uht(t_{n,j}) = \boldsymbol{U}_{n,h}^j,\qquad
p_{h,\tau}(t_{n,j}) = P_{n,h}^j, \qquad j=1,\ldots,k+1,
\]
with $t_{n,j}=T_n(\hat{t}_j)$, $j=1,\ldots,k+1$. Taking into consideration
that the Gau\ss--Radau formula $Q_n$ is exact for polynomials up to degree
$2k$, the particular choices
\begin{align*}
q_{h,\tau} = 0, \quad \vht
& = \frac{1}{\widehat{\omega}_j}\varphi_{n,j}(t)\bv_h
\text{ with arbitrary } \bv_h\in \vv_h,
\intertext{and}
\vht = \boldsymbol{0}, \quad q_{h,\tau}
& = \frac{1}{\widehat{\omega}_j}\varphi_{n,j}(t)q_h
\text{ with arbitrary } q_h\in Q_h,
\end{align*}
for the test functions lead to the following system of nonlinear
equations:
\medskip

Find the coefficients $(\boldsymbol{U}_{n,h}^j,P_{n,h}^j)\in \vv_h\times Q_h$,
$j=1,\ldots k+1$, such that 
\begin{multline}
\qquad\sum_{j=1}^{k+1} \alpha_{ij} \big( \boldsymbol{U}_{n,h}^j, \bv_h\big) +
\frac{\tau_n}{2}
A_h\big( (\boldsymbol{U}_{n,h}^i,P_{n,h}^i),(\bv_h,q_h)\big)
+ \frac{\tau_n}{2} n(\boldsymbol{U}_{n,h}^i, \boldsymbol{U}_{n,h}^i, \bv_h)\\
= \beta_i \big( \boldsymbol{U}_{n,h}^0,\bv_h\big)
+ \frac{\tau_n}{2} \big\langle\ff(t_{n,i}), \bv_h\big\rangle \qquad
\label{eq:fdp_algebraic}
\end{multline}
for $i=1,\ldots, k+1$ and for all $(\bv_h, q_h)\in \vv_h\times Q_h$ where 
\begin{align*}
\alpha_{ij} \coloneqq  \widehat{\varphi}_j'(\hat{t}_i) + \beta_i \widehat{\varphi}_j(-1),\qquad
\beta_i \coloneqq  \frac{1}{\widehat{\omega}_i}
\widehat{\varphi}_i(-1),
\end{align*}
see~\cite{MS11}. The initial condition $\boldsymbol{U}_{n,h}^0$ on $I_n$
is given by
\[
\boldsymbol{U}_{n,h}^0 \coloneqq 
\begin{cases}
j_h \bu_0, & n=1,\\
\boldsymbol{U}_{n-1,h}^{k+1}, & n>1.
\end{cases}
\]
Note that no initial pressure is required.
\subsection{Velocity estimates: stability and convergence}
In this section, we study the stability properties and the error analysis of
the fully discrete scheme~\eqref{eq:fdp_one_form} with respect to the
velocity. To this end, we exploit the skew-symmetry~\eqref{eq:skewprop} of the
nonlinear term $n$.

Following ideas presented in~\cite{FKNP11}, we define for any function
space $W$ the operators $\pi_n:C(I_n,W)\to\mathbb{P}_k(I_n,W)$,
$n=1,\dots,N$, by
\begin{equation}
\label{eq:lagrange_def}
(\pi_n w)(t_{n,i}) = w(t_{n,i}) \frac{\tau_n}{t_{n,i}-t_{n-1}} = w(t_{n,i})
\widehat{s}_i,\quad
\widehat{s}_i = \frac{2}{\hat{t}_i+1},\qquad i=1,\dots,k+1.
\end{equation}
Hence, $\pi_n w$ is the Lagrange interpolant of the function
$t\mapsto \tau_nw(t)/(t-t_{n-1})$ with respect to the Gau\ss--Radau points
$t_{n,i}\in I_n$, $i=1,\dots,k+1$. Note that $\widehat{s}_1>
\widehat{s}_2>\dots>\widehat{s}_{k+1}=1$. We set
\begin{equation}
\label{eq:lagrange_def2}
\widehat{s}:=\max_{1\le i\le k+1}\widehat{s}_i=\widehat{s}_1
\end{equation}
to shorten some notation.

Provided that $W$ is either a subspace of $L^2(\Omega)$ or
$L^2(\Omega)^d$, the mapping
\begin{equation*}
w\mapsto \|w\|_n \coloneqq \left(
\frac{\tau_n}{2} \sum_{i=1}^{k+1} \widehat{\omega}_i\widehat{s}_i
\|w(t_{n,i})\|_0^2
\right)^{1/2}
\end{equation*}
gives a norm on $\mathbb{P}_k(I_n,W)$ satisfying
\begin{equation}
\label{eq:fs_est1}
\big\|(\pi_n w)(t_{n-1}^+)\big\|_0^2 \le \frac{C_1}{\tau_n} \|w\|_n^2
\qquad\forall w\in\mathbb{P}_k(I_n,W)
\end{equation}
where the fixed constant $C_1$ depends on the polynomial degree $k$ but
is independent of $\tau_n$ and $w\in\mathbb{P}_k(I_n,W)$. Moreover,
we have
\begin{equation}
\label{eq:fs_est2}
\Qn{\|\pi_n w\|_0^2} \le \widehat{s} \|w\|_n^2,
\qquad\
\Qn{\|w\|_0^2} \le \|w\|_n^2\le \widehat{s} \Qn{\|w\|_0^2}
\end{equation}
for all $w\in C(I_n,W)$.
We refer to Lemmata~3 and~5 by~\cite{FKNP11} for details.

The following result provides the stability of the fully discrete velocity. 
\begin{theorem}\label{thm:stab_velocity}
	Let $(\uht,\pht) \in \xx_k \times Y_k$ be the solution of the fully
	discrete scheme~\eqref{eq:fdp_one_form}. Furthermore, we assume $\ff\in
	C(L^2)$. Then, the estimate
	\begin{multline}
	\label{eq:stab_velocity}
	\|\uht(t_n^-)\|_0^2 + 2 \sum_{i=1}^nQ_i\left[\nu\|\nabla\uht\|_0^2
	+ S_h(\uht,\uht)\right] \\
	\le
	\exp(8C_1\,t_n) \left(\|j_h\bu_0\|_0^2
	+ \sum_{i=1}^{n} (1+16\widehat{s}\,\tau_i^2)
	Q_i\left[\|\ff\|_0^2\right]\right)
	\end{multline}
	with $C_1$ from~\eqref{eq:fs_est1} and $\widehat{s}$
	from~\eqref{eq:lagrange_def2} holds true for all $n=1,\ldots,N$.
\end{theorem}
\begin{proof}
	Setting $(\vht,\qht) = (\uht,\pht)$ in~\eqref{eq:fdp_one_form},
	the exactness of the quadrature rule applied to $(\uht',\uht)$ and
	the application of~\eqref{eq:skewprop} and~\eqref{eq:Ah_coer} result in
	\begin{multline}
	\label{eq:stab_n1}
	\frac{1}{2}\|\uht(t_n^-)\|_0^2 + \Qn{\nu\|\nabla\uht\|_0^2
		+ S_h(\uht,\uht)}\\
	\le \frac{1}{2} \|\uht(t_{n-1}^-)\|_0^2
	+ \frac{1}{2} \Qn{\|\ff\|_0^2} +
	\frac{1}{2} \Qn{\|\uht\|_0^2}.
	\end{multline}
	Setting $(\vht,\qht) = (\pi_n \uht, \pi_n \pht)$
	in~\eqref{eq:fdp_one_form}, we get
	\begin{multline}
	\Qn{(\uht', \pi_n \uht)} + \left(\uht (t_{n-1}^+), \pi_n \uht(t_{n-1}^+)\right) 
	+ \Qn{A_h\big((\uht,\pht), (\pi_n \uht, \pi_n\pht) \big)}\\ 
	+\Qn{n(\uht, \uht, \pi_n \uht)}
	= \left(\uht(t_{n-1}^-), \pi_n \uht(t_{n-1}^+)\right)
	+ \Qn{(\ff, \pi_n \uht)}.
	\label{eq:stab_n2}
	\end{multline}
	The nonlinear term $n$ vanishes in each quadrature point of $Q_n$ due
	to~\eqref{eq:lagrange_def} and the skew-symmetric
	property~\eqref{eq:skewprop}. Using the exactness of the quadrature
	formula, the first two terms on the left-hand side
	of~\eqref{eq:stab_n2} are bounded from below by
	\begin{align*}
	\frac{1}{2}\left(\|\uht(t_{n}^-)\|_0^2
	+ \frac{1}{\tau_n} \|\uht\|_n^2\right)
	\le \Qn{(\uht',\pi_n \uht)} + \left(\uht(t_{n-1}^+), \pi
	\uht(t_{n-1}^+)\right),
	\end{align*}
	see~(4.24) in~\cite{FKNP11} for more details.
	Since $\widehat{s}_i \ge 1$, the use of~\eqref{eq:lagrange_def}
	and the coercivity~\eqref{eq:Ah_coer} give for the third term on the
	left-hand side of~\eqref{eq:stab_n2} the estimate
	\begin{align*}
	\Qn{A_h\big((\uht,\pht), (\pi_n \uht, \pi_n\pht) \big)}
	&= \frac{\tau_n}{2} \sum_{j=1}^{k+1} \widehat{\omega}_j \widehat{s}_j
	A_h\big((\uht(t_{n,j}),\pht(t_{n,j})),
	(\uht(t_{n,j}),\pht(t_{n,j}))\big)\\
	&\ge \Qn{\nu \|\nabla \uht\|_0^2 + S_h(\uht,\uht)}\ge 0.
	\end{align*}
	Inserting these estimates into~\eqref{eq:stab_n2} and using
	Cauchy--Schwarz' and Young's inequalities as well as~\eqref{eq:fs_est1}
	and~\eqref{eq:fs_est2} for the right-hand side of~\eqref{eq:stab_n2}, we get 
	\begin{align*}
	\frac{1}{2}\|\uht(t_{n}^-)\|_0^2
	+ \frac{1}{2\tau_n} \|\uht\|_n^2
	&\le \|\uht(t_{n-1}^-)\|_0 \; \|\pi_n \uht(t_{n-1}^+)\|_0
	+ \Qn{\|\ff\|_0^2}^{1/2}\Qn{\|\pi_n\uht\|_0^2}^{1/2}\\
	&\le C_1 \|\uht(t_{n-1}^-)\|_0^2
	+ \frac{1}{4\tau_n} \|\uht\|_n^2
	+ 2\widehat{s}\,\tau_n \Qn{\|\ff\|_0^2}
	+ \frac{1}{8\tau_n} \|\uht\|_n^2.
	\end{align*}
	Putting the terms with $\|\uht\|_n^2$ to the left-hand side and skipping
	non-negative contributions there, we arrive at
	\begin{align*}
	\frac{1}{8\tau_n} \|\uht\|_n^2
	\le C_1 \|\uht(t_{n-1}^-)\|_0^2
	+ 2\widehat{s}\, \tau_n \Qn{\|\ff\|_0^2}
	\end{align*}
	which by~\eqref{eq:fs_est2} gives
	\begin{align*}
	\Qn{\|\uht\|_0^2} \le \|\uht\|_n^2 \le
	8 C_1\tau_n \|\uht(t_{n-1}^-)\|_0^2 +
	16 \widehat{s}\, \tau_n^2 \Qn{\|\ff\|_0^2}.
	\end{align*}
	Inserting this in \eqref{eq:stab_n1} leads to 
	\begin{equation*}
	\frac{1}{2}\|\uht(t_n^-)\|_0^2 + \Qn{\nu\|\nabla\uht\|_0^2
		+ S_h(\uht,\uht)}
	\le  \frac{1}{2}\big(1 + 8 C_1 \tau_n \big) \|\uht(t_{n-1}^-)\|_0^2
	+ \frac{1}{2} \big(1+16\widehat{s}\,\tau_n^2\big) \Qn{\|\ff\|_0^2}.
	\end{equation*}
	The application of a discrete version of the Gronwall lemma to this
	estimate concludes the proof.
\end{proof}

To prepare an error estimate for the  velocity, more notation is needed.
Let $w$ be a time-continuous function. We define its Gau\ss--Radau
interpolant $\widetilde{w}$ as
\[
\widetilde{w}|_{I_n}(t) \coloneqq  \sum_{j=1}^{k+1} w(t_{n,j}) \varphi_{n,j}(t),
\]
with $\varphi_{n,j}$ given in~\eqref{phinj}. Moreover, we set
$\widetilde{w}_0^- \coloneqq  w_0^-$.
Note that $\widetilde{\bu}$ and $\widetilde{p}$ will be on each time interval $I_n$,
$n=1,\dots,N$, polynomials of degree less than or equal to $k$ with values
in $\vv$ and $Q$ which coincide with $\bu$ and $p$ in all quadrature points
$t_{n,i}$. Furthermore, we define $w^I$ on each $I_n$
as the Lagrange interpolant of $w$ with respect to the nodes
$t_{n-1},t_{n,1},\dots,t_{n,k+1}$. Hence, $w^I$ is a time-continuous,
piecewise polynomial of degree less than or equal to $k+1$.

Using multiple times that the interpolants $\widetilde{\w}$ and $\w^I$
coincide in all quadrature points $t_{n,i}$, integration by parts in time
and the exactness of the quadrature rule $Q_n$ for polynomials of degree
less than or equal to $2k$, we obtain
\begin{multline}\label{eq:rel_interpl}
\Qn{\big((\w^I-\widetilde{\w})',\vht\big)}
= \intin \big((\w^I-\widetilde{\w})',\vht\big)\\
\begin{aligned}
& = - \intin (\w^I - \widetilde{\w}, \vht')
+ \big((\w^I-\widetilde{\w})(t_n^-), \bv_n^-\big)  
- \big((\w^I-\widetilde{\w})(t_{n-1}^+), \bv_{n-1}^+\big)\\
& = -\Qn{\big(\w^I-\widetilde{\w},\vht'\big)}
- \big(\widetilde{\w}(t_{n-1}^-) - \widetilde{\w}(t_{n-1}^+),
\bv_{n-1}^+\big)\\
& = \big([\widetilde{\w}]_{n-1}, \bv_{n-1}^+\big)
\end{aligned}
\end{multline}
for all $\vht\in \xx_k$.

Moreover, the standard interpolation theory leads to the error estimates
\begin{align}\label{t3}
\sup_{0\le t\le T}|w^{(i)}(t)-\widetilde{w}^{(i)}(t) |_{j,p}
&\le C\tau^{k+1-i}\sup_{0\le t\le T}|w^{(k+1)}(t)|_{j,p},\\ 
\intin|w^{(i)}(t)-\widetilde{w}^{(i)}(t)|_{j,p}^2 
&\le C\tau_n^{2(k+1-i)} \intin|w^{(k+1)}(t)|_{j,p}^2,\label{t4}\\
\sup_{0\le t\le T} |\bu(t) - \bu^I(t)|_{j,p}
&\le C\tau^{k+2} \sup_{0\le t\le T} |\bu^{(k+2)}(t)|_{j,p}\label{t5}
\end{align}
with $i,j=0,1$ and $p\in[1,\infty]$.

\begin{assumption}\label{reg:assumpt}
	The a priori error analysis below assumes
	\begin{equation*}
	\bu \in C^1\big(H^{r+1}(\Omega)^d\big) \cap C^{k+2}\big(H^1(\Omega)^d\big)
	\quad \textrm{and} \quad
	p \in C\big(H^r(\Omega)\big) \cap H^{k+1}\big(H^1(\Omega)\big)
	\end{equation*}
	for the solution $(\bu,p)$ of the Navier-Stokes
	equations~\eqref{eq:weak_form}.
\end{assumption}
The subsequent analysis is based on exploiting properties of discretely
divergence-free functions. We define
\begin{alignat*}{4}
\eht & \coloneqq \bu-\uht,  \qquad & \bxt & \coloneqq  j_h\tu - \uht, \qquad & \be_\tau & \coloneqq  \tu - \bu, \qquad &
\beh & \coloneqq  j_h\bu - \bu,\\
\rhoht & \coloneqq p-\pht, \qquad &
\tht & \coloneqq  i_h\tp - \pht, \qquad & \phit & \coloneqq  \tp - p, \qquad &
\phih & \coloneqq  i_h p - p,
\end{alignat*}
where $\bxt\in\xx_k^{\textrm{div}}$ and $\tht\in Y_k$ are fully
discrete velocity and pressure functions, respectively. Furthermore,
the error splittings
\begin{alignat}{2}
j_h\tu - \bu  &= j_h\be_\tau + \beh,\qquad &
i_h \tp- p  &= i_h\phit + \phih,
\label{eq:error_splitting}\\
\bu-\uht & = \bxt-j_h \bet - \beh, \qquad & p-\pht & = \tht - i_h \phit -
\phih\nonumber
\end{alignat} 
hold true. The error equation
\begin{align}\label{eq:error_equation}
\Qn{(\bxt', \vht)} &+ \Qn{A_h\big((\bxt,\tht), (\vht,\qht)\big)} + \left([\bxt]_{n-1}, \vht(t_{n-1}^+)\right) \nonumber\\
&= \Qn{(j_h\tu'-\bu', \vht)} + \Qn{A_h\big((j_h\tu-\bu, i_h \tp- p), (\vht,\qht)\big)}\\
&\qquad  +\Qn{S_h(\bu, \vht)}  + \left([j_h\tu-\bu]_{n-1}, \vht(t_{n-1}^+)\right)\nonumber \\
&\qquad\qquad - \Qn{n\left(\bu, \bu, \vht\right) - n\left(\uht, \uht, \vht\right)}\nonumber
\end{align}
is obtained by using the fully discrete problem \eqref{eq:fdp_one_form}
and property~\eqref{eq:new_weak_form} of the continuous problem.

The difference of the nonlinear terms is estimated as follows: 
\begin{lemma}\label{lem:nonlinear_estimate}
	Let $(\bu,p)$ with $\bu|_{I_n}\in C(I_n,W^{1,\infty})$ and $(\uht,\pht)$
	be the solutions of~\eqref{eq:weak_form} and~\eqref{eq:fdp_one_form},
	respectively. Furthermore, let
	$\vht\in\mathbb{P}_k(I_n,\vv_h^{\textrm{div}})$ with $\vht(t_{n,i}) =
	\gamma_i\bxt(t_{n,i})$, $i=1,\dots,k+1$, where $\gamma_i\in\mathbb{R}$ are
	positive constants such that
	\begin{equation}\label{eq:gammai}
	\max_{1\le i\le k+1}\gamma_i^{-1}\le 1.
	\end{equation}
	Then, there exists a positive constant $C$ depending only on $j_h$
	such that the estimate
	\begin{multline}\label{eq:nonlinear_estimate}
	\big|\Qn{n\left(\bu, \bu, \vht\right) - n\left(\uht, \uht, \vht
		\right)}\big| \\
	\le C \Qn{ \|\bu - j_h \bu \|_1^2} \|\bu\|_{C(I_n,W^{1,\infty})}^2
	+ C_n(\bu) \Qn{\|\vht\|_0^2} + \frac{1}{8} \Qn{S_h(\bxt,\bxt)}
	\end{multline}
	holds true where
	\begin{align}\label{eq:C1}
	C_n(\bu) \coloneqq 2 + \frac{C_d^2}{2}
	\|j_h\bu\|_{C(I_n,L^{\infty})}^2
	+ \|\nabla j_h\bu\|_{C(I_n,L^{\infty})}
	\end{align}
	depends on $\|\bu\|_{C(I_n,W^{1,\infty})}$ only.\\
\end{lemma}

\begin{proof}
	Having in mind that $\bu$ and $\widetilde{\bu}$ coincide in all quadrature
	points, we split the difference of the nonlinear terms as follows
	\begin{multline*}
	\Qn{n\left(\bu, \bu, \vht \right) - n\left(\uht, \uht, \vht \right)}\\
	\begin{aligned}
	&=\Qn{n\left(\bu-\uht, \bu,  \vht \right)
		+ n\left(\uht,\bu-\uht, \vht \right)} \\
	& = \Qn{n\left(\bu - j_h \bu,\bu,\vht\right)}
	+ \Qn{n\left(\bxt,j_h\bu,\vht\right)}
	+ \Qn{n\left(j_h\bu,\bu - j_h\bu,\vht\right)}
	\end{aligned}
	\end{multline*}
	where we used
	\[
	n\big(\uht(t_{n,i}),\bxt(t_{n,i}),\vht(t_{n,i})\big) =
	\gamma_i n\big(\uht(t_{n,i}),\bxt(t_{n,i}),\bxt(t_{n,i})\big) = 0
	\]
	due to $\vht(t_{n,i})=\gamma_i\bxt(t_{n,i})$ and~\eqref{eq:skewprop}.
	Applying generalised H\"older's inequalities followed by Young's
	inequalities, we get with~\eqref{eq:partint} and~\eqref{eq:QnCSU}
	\begin{align*}
	\Big|\Qn{n\left(\bu - j_h\bu,\bu,\vht\right)}\Big|
	& \le \Qn{\|\bu - j_h\bu\|_0 \|\nabla\bu\|_{L^{\infty}} \|\vht\|_0}
	+ \frac{1}{2} \Qn{\|\Div(\bu - j_h\bu)\|_0 \|\bu\|_{L^{\infty}}
		\|\vht\|_0}\\
	& \le \frac{1}{2} \Qn{\|\bu-j_h\bu\|_0^2}
	\|\nabla\bu\|_{C(I_n,L^{\infty})}^2
	+ \frac{1}{8} \Qn{\|\Div(\bu-j_h\bu)\|_0^2}
	\|\bu\|_{C(I_n,L^{\infty})}^2\\
	&\qquad+ \Qn{\|\vht\|_0^2}.
	\end{align*}
	The choice of $\vht$, Lemma~\ref{lem:discdiv}, conditions~\eqref{eq:gammai}
	and~\eqref{eq:partint}, and Young's inequality yield the bound
	\begin{align*}
	\Big|\Qn{n(\bxt,j_h\bu,\vht)}\Big|
	\le \frac{1}{8} \Qn{S_h(\bxt,\bxt)}
	+ \left(\frac{C_d^2}{2}
	\|j_h\bu\|_{C(I_n,L^{\infty})}^2
	+ \|\nabla j_h\bu\|_{C(I_n,L^{\infty})}\right)
	\Qn{\|\vht\|_0^2},
	\end{align*}
	The third nonlinear term is estimated similarly to the first
	one. We obtain	
	\begin{align*}
	\Big|\Qn{n\left(j_h\bu,\bu - j_h\bu,\vht\right)}\Big|
	& \le \Qn{\|j_h\bu\|_{L^{\infty}} \|\nabla(\bu - j_h\bu)\|_0
		\|\vht\|_0}\\
	& \qquad +\frac{1}{2}\Qn{ \,\|\Div j_h\bu\|_{L^{\infty}} 
		\|\bu - j_h\bu\|_0 \|\vht\|_0}\\
	& \le \frac{1}{2}\Qn{\|\nabla(\bu - j_h\bu\|_0^2}
	\|j_h\bu\|_{C(I_n,L^{\infty})}^2\\
	& \qquad + \frac{1}{8}\Qn{\|\bu - j_h\bu\|_0^2}
	\|\Div j_h\bu\|_{C(I_n,L^{\infty})}^2
	+\Qn{\|\vht\|_0^2}. 
	\end{align*}
	The statement of this lemma follows by collecting the above statements
	and applying~\eqref{j1} for estimating in the last inequality the terms
	involving $j_h\bu$.
\end{proof}

We define
\begin{align}
\label{eq:En}
E_n  \coloneqq \sqrt{\tau_n} \;\Bigg\{ 
\tau_n^{k+1} \|\bu\|_{C^{k+2}(I_n,H^1)}
+ h^r \Big(\|\bu\|_{C^1(I_n,H^{r+1})}
\big(\|\bu\|_{C(I_n,W^{1,\infty})} + 1 + \sqrt{\nu}\big)
+ \|p\|_{C(I_n,H^r)} \Big)
\Bigg\}
\end{align}
to shorten the notation of our error estimates.

\begin{lemma}\label{lem:fdp_disc_velo_est}
	Let the spaces $\vv_h$ and $Q_h$ satisfy the discrete inf-sup
	condition~\eqref{eq:disc_inf_sup}. Suppose
	assumptions~\ref{assmption_a1}, \ref{assmption_a2}, and $\mu_K \sim 1$
	for all $K\in \mathcal{T}_h $. 
	Let $(\bu,p)$ and $(\uht,\pht)$ be the solutions of the continuous 
	problem~\eqref{eq:weak_form} and the fully discrete
	problem~\eqref{eq:fdp_one_form}, respectively. Furthermore, assume that 
	the solution $(\bu,p)$ satisfies the regularity
	assumption~\ref{reg:assumpt}. Then, the estimate
	\begin{multline}
	\|\bxt(t_n^-)\|_0^2
	+  \|[\bxt]_{n-1}\|_0^2 +
	\Qn{\nu \|\nabla \bxt\|_0^2 + S_h(\bxt,\bxt)} + \Qn{\|\bxt\|_0^2}\\
	\le \|\bxt(t_{n-1}^-)\|_0^2
	+ C\, E_n^2 + 2\big(1+ C_n(\bu)\big) \Qn{\|\bxt\|_0^2}
	\label{eq:fdp_disc_velo_est}
	\end{multline}
	holds true where $C_n(\bu)$ is given in~\eqref{eq:C1}. The constant $C$ is
	independent of $\tau$, $h$, and $\nu$.
\end{lemma}

\begin{proof}
	Recall $\bxt=j_h\tu - \uht \in \xx_k^{\textrm{div}}$ and
	$\tht=i_h \tp -\pht \in Y_k$.
	Setting $(\vht,\qht) = (\bxt,\tht) $ in \eqref{eq:error_equation} and using
	\begin{equation*}
	\Qn{(\bxt',\bxt)} + \left([\bxt]_{n-1},\bxt(t_{n-1}^+) \right)
	= \frac{1}{2} \|\bxt(t_n^-)\|_0^2 - \frac{1}{2} \|\bxt(t_{n-1}^-)\|_0^2
	+ \frac{1}{2} \|[\bxt]_{n-1}\|_0^2
	\end{equation*}
	we obtain by~\eqref{eq:lps}
	\begin{multline}
	\frac{1}{2} \|\bxt(t_n^-)\|_0^2 - \frac{1}{2} \|\bxt(t_{n-1}^-)\|_0^2
	+ \frac{1}{2} \|[\bxt]_{n-1}\|_0^2 +
	\Qn{\nu \|\nabla \bxt\|_0^2 + S_h(\bxt,\bxt)}\\
	\le \mathcal{N}_1+ \mathcal{N}_2 
	+ \Qn{ S_h(\bu,\bxt)}
	-\Qn{n(\bu,\bu,\bxt) - n(\uht,\uht,\bxt)}
	\label{eq:fdp_vel_err}
	\end{multline}
	with
	\begin{align*}
	\mathcal{N}_1
	& \coloneqq \Qn{\left(j_h\bet', \bxt\right)}
	+ \left([j_h\bet]_{n-1} , \bxt(t_{n-1}^+)\right) +
	\Qn{A_h\big((j_h\bet,i_h\phit),(\bxt,\tht) \big)}
	\intertext{and}
	\mathcal{N}_2
	& \coloneqq \Qn{\left(\beh', \bxt\right)}
	+ \left([\beh]_{n-1} , \bxt(t_{n-1}^+)\right)
	+ \Qn{A_h\big((\beh,\phih),(\bxt,\tht) \big)}.
	\end{align*}
	We will bound the terms on the right-hand side of~\eqref{eq:fdp_vel_err}
	separately. Taking into consideration that $(\bu,p)$ and $(\tu,\tp)$
	coincide in all quadrature points, we have $\bet(t_{n,i})=\boldsymbol{0}$ in all
	quadrature points. This gives
	\begin{equation}
	\mathcal{N}_1
	=\Qn{ \left(j_h \be_{\tau}', \bxt\right)} +
	\left([j_h\be_{\tau}]_{n-1} , \bxt(t_{n-1}^+)\right)  =\Qn{\left(j_h(\bu^I - \bu)', \bxt \right)} 
	\label{eq:n1}
	\end{equation}
	where we have used~\eqref{eq:rel_interpl} and the fact that the
	interpolation operators in time and space commute. The
	Cauchy--Schwarz inequality and Young's inequality give
	\begin{align*}
	\mathcal{N}_1
	&\le \Qn{\|j_h(\bu^I-\bu)'\|_0^2}^{1/2} \Qn{\|\bxt\|_0^2}^{1/2}
	\le  \frac{2}{3}  \Qn{\|j_h(\bu^I-\bu)'\|_0^2}
	+ \frac{3}{8} \Qn{\|\bxt\|_0^2}.
	\end{align*}
	Note that the jump term in $\mathcal{N}_2$ vanishes due to the continuity of
	$\beh=j_h \bu - \bu$ in time.
	Hence, we get
	\begin{equation*}
	\mathcal{N}_2
	=Q_n\Big[ \left(\beh', \bxt\right) 
	+ \nu\left(\nabla \beh,\nabla \bxt \right) -(\phih,\Div \bxt) +S_h(\beh,\bxt)\Big]
	\end{equation*}
	since $(\Div \be_h,\tht)=0$ due to~\eqref{discdiv}.
	
	Adapting the techniques used to bound the similar terms in the semi-discrete
	analysis, we get
	\begin{align}
	\mathcal{N}_2
	&\le  \Qn{2\|\beh'\|_0^2 + \frac\nu2 \|\nabla \beh\|_0^2 +2 S_h(\beh,\beh)
		+ \frac{2d}{\mu_h^{\textrm{min}}}\|\phih\|_0^2}
	+ \frac{\nu}{2} \Qn{ \|\nabla \bxt\|_0^2}
	\nonumber\\
	&\quad + \frac{1}{4}\Qn{S_h(\bxt,\bxt)}+\frac{1}{8} \Qn{\|\bxt\|_0^2}.
	\label{eq:n2}
	\end{align}
	The estimate for the third term on the right-hand side
	of~\eqref{eq:fdp_vel_err} uses the Cauchy--Schwarz-like
	estimate~\eqref{eq:QnCSU} and Young's inequality to get
	\[
	\Qn{ S_h(\bu, \bxt)} \le 2 \Qn{S_h(\bu,\bu)} + \frac{1}{8}\Qn{S_h(\bxt,\bxt)}.
	\]
	Collecting the above estimates in~\eqref{eq:fdp_vel_err}, using
	estimate~\eqref{eq:nonlinear_estimate} for the difference of the nonlinear
	terms, and contributing similar norm terms to the left-hand side, we obtain
	the statement of this lemma.
\end{proof}

\begin{lemma}Let the assumptions of Lemma~\ref{lem:fdp_disc_velo_est} hold.
	Then, the bound
	\label{lem:veloc_l2_estim}
	\begin{equation}
	\Qn{\|\bxt\|_0^2}
	\le 8 C_1 \tau_n\|\bxt (t_{n-1}^-)\|_0^2 + C \tau_nE_n^2
	\label{eq:velo_l2_estim}
	\end{equation}
	holds true provided that
	\begin{equation}
	\label{eq:restrict_tau}
	\tau_n \le \frac{1}{8\big(1+C_n(\bu)\widehat{s}\big)}
	\end{equation}
	is fulfilled.
\end{lemma}
\begin{remark}
	Note condition~\eqref{eq:restrict_tau} is not a CFL conditions since the
	bounds depend only on the problem data and the order of the dG method, but
	not on the spatial mesh size $h$.
\end{remark}
\begin{proof}
	Substituting $(\vht,\qht) = (\pi_n \uht, \pi_n \qht)$ in the error
	equation~\eqref{eq:error_equation} and proceeding for the left-hand side 
	as in the proof of the stability estimate, we arrive at 
	\begin{equation}
	\label{eq:l2_estimate_erreq}
	\frac{1}{2} \|\bxt (t_n^-)  \|_0^2
	+ \frac{1}{2\tau_n} \|\bxt\|_n^2
	+ \Qn{\nu \|\nabla \bxt\|_0^2 + S_h(\bxt,\bxt)}
	\le J_1 + J_2 + J_3 + J_4
	\end{equation}
	with
	\begin{align*}
	J_1 &\coloneqq \big(\bxt (t_{n-1}^-), \pi_n \bxt(t_{n-1}^+)\big),\\
	J_2 &\coloneqq \Qn{((j_h\tu - \bu)',\pi_n \bxt )}
	+ \big([j_h\tu - \bu]_{n-1}, \pi_n \bxt(t_{n-1}^+)\big),\\
	J_3 &\coloneqq \Qn{A_h\big((j_h\tu - \bu, i_h\tp-p),
		(\pi_n\bxt,\pi_n\tht) \big) +S_h( \bu, \pi_n \bxt)},\\
	J_4 &\coloneqq \Qn{n(\uht, \uht, \pi_n\bxt)
		- n\left(\bu,\bu,\pi_n\bxt\right)}.
	\end{align*}
	We shall consider the terms on the right-hand side separately. We get
	for the first term
	\begin{equation*}
	J_1 \le \|\bxt (t_{n-1}^-)\|_0\; \|\pi_n \bxt(t_{n-1}^+)\|_0
	\le C_1\|\bxt (t_{n-1}^-)\|_0^2 
	+ \frac{1}{4 \tau_n} \| \bxt\|_n^2
	\end{equation*}
	by using the Cauchy--Schwarz inequality, the bound~\eqref{eq:fs_est1}, and
	Young's inequality.
	Taking additionally into consideration that $(\bu,p)$ and $(\tu,\tp)$
	coincide in all quadrature points, the Cauchy--Schwarz inequality,
	\eqref{eq:fs_est2}, and Young's inequality give
	\begin{align*}
	J_2 &= \Qn{(j_h\bet' + \beh', \pi_n \bxt)}
	+ \left([j_h \bet]_{n-1}, \pi_n \bxt (t_{n-1}^+) \right)\\
	& = \Qn{( \beh', \pi_n \bxt)} +\Qn{(j_h(\bu^I - \bu)', \pi_n \bxt))} 
	\le \frac{\widehat{s}}{2} \Qn{\|\beh'\|_0^2}
	+ \frac{\widehat{s}}{2}\Qn{\|j_h(\bu^I - \bu)'\|_0^2}
	+ \|\bxt\|_n^2
	\end{align*}
	where~\eqref{eq:rel_interpl} and the commutation of temporal and spatial
	interpolations were exploited. Using the definition of $A_h$, the error
	splitting, and the same arguments as in the proof of
	Thm.~\ref{thm:semi_disc_vel_estimates}, we get
	\begin{align*}
	J_3 &\le \frac{\widehat{s}^2}{2} \Qn{ \nu\|\nabla \beh\|_0^2 }
	+ \frac{2\widehat{s}^2\,d}{\mu_h^{\textrm{min}}}
	\Qn{\|\varphi_h\|_0^2}
	+ 2\widehat{s}^2 \Qn{S_h(\beh,\beh)}
	+ 2\widehat{s}^2 \Qn{S_h(\bu,\bu)} 
	\\
	&\qquad + \frac{1}{2} \Qn{\nu\|\nabla \bxt\|_0^2}
	+ \frac{3}{8} \Qn{S_h(\bxt,\bxt)}.
	\end{align*}
	Using definition~\eqref{eq:lagrange_def} together with
	$\widehat{s}_i\ge 1$, $i=1,\dots,k+1$,
	Lemma~\ref{lem:nonlinear_estimate} can be applied
	with $\vht=\pi_n \bxt$. Exploiting~\eqref{eq:fs_est2}, we obtain
	\begin{align*}
	J_4
	& \le C \Qn{ \|\bu - j_h \bu \|_1^2} \|\bu\|_{C(I_n,W^{1,\infty})}^2
	+ C_n(\bu) \Qn{\|\pi_n\bxt\|_0^2} + \frac{1}{8} \Qn{S_h(\bxt,\bxt)}\\
	& \le C \Qn{ \|\beh\|_1^2} \|\bu\|_{C(I_n,W^{1,\infty})}^2
	+ C_n(\bu) \widehat{s} \|\bxt\|_n^2
	+ \frac{1}{8} \Qn{S_h(\bxt,\bxt)}.
	\end{align*}
	Using the bounds for $J_1,\dots,J_4$ in~\eqref{eq:l2_estimate_erreq} gives
	after contributing similar term to the left-hand side the estimate
	\begin{align*}
	\frac{1}{2} & \|\bxt (t_n^-)  \|_0^2
	+ \left( \frac{1}{4\tau_n} - 1 - \widehat{s} C_n(\bu) \right)\|\bxt\|_n^2
	+ \frac{1}{2}\Qn{\nu \|\nabla \bxt\|_0^2 + S_h(\bxt,\bxt)}\\
	& \le C_1\|\bxt (t_{n-1}^-)\|_0^2 
	+ \frac{\widehat{s}}{2} \Qn{\|\beh'\|_0^2}
	+ \frac{\widehat{s}}{2}\Qn{\|j_h(\bu^I - \bu)'\|_0^2}
	+ \Qn{ \|\beh \|_1^2}\\
	&\quad + \frac{\widehat{s}^2\,d}{2} \Qn{ \nu\|\nabla \beh\|_0^2 } 
	+ \frac{2\widehat{s}^2}{\mu_h^{\textrm{min}}}
	\Qn{\|\varphi_h\|_0^2}
	+ 2\widehat{s}^2 \Qn{S_h(\beh,\beh)}
	+ 2\widehat{s}^2 \Qn{S_h(\bu,\bu)} .
	\end{align*}
	Exploiting the condition~\eqref{eq:restrict_tau}, the proof is completed by
	applying stability and error estimates for the interpolations operators
	in space and time.
\end{proof}

\begin{theorem}\label{thm:fdp_velocity_estimates}
	Assume that the finite element spaces satisfy the discrete inf-sup
	condition~\eqref{eq:disc_inf_sup}.
	Suppose the assumptions~\ref{assmption_a1}, \ref{assmption_a2},
	and $\mu_K\sim 1$ for all $K\in\mathcal{T}_h$. Let $(\bu,p)$ be the
	solution of the continuous problem~\eqref{eq:weak_form} and $(\uht,\pht)$
	be the solution of the fully discrete scheme~\eqref{eq:fdp_one_form} with
	$\bu_{0,h}=j_h\bu_0$. Furthermore, assume that the solution $(\bu,p)$
	satisfies the regularity assumption~\ref{reg:assumpt}.
	Then, there exists a constant $C$ independent of $h$, $\nu$, and $\tau$
	such that the error estimate
	\begin{multline}
	\label{eq:velo_est}
	\|(\eht)(t_m^{-})\|_0^2  + 
	\sum_{n=1}^m \Qn{\nu \|\nabla \eht\|_0^2 + S_h(\eht,\eht)
		+\|\eht\|_0^2}
	+ \sum_{n=1}^m \|[\eht]_{n-1}\|_0^2\\
	\begin{aligned}
	\le C \exp\big(16C_1(1+C(\bu))t_n\big) (1+C(\bu)) \sum_{m=1}^n E_m^2
	\end{aligned}
	\end{multline}
	holds true where
	\begin{equation}
	\label{eq:Cu}
	C(\bu) = \max_{n=1,\dots,N} C_n(\bu)
	\end{equation}
	is independent of $\tau_n$, $h$, and $\nu$.
\end{theorem}

\begin{proof}
	Combining the estimates of Lemmata~\ref{lem:fdp_disc_velo_est}
	and~\ref{lem:veloc_l2_estim}, we get for $n=1,\ldots,N$
	\begin{multline*}
	\|\bxt(t_n^-)\|_0^2 +  \|[\bxt]_{n-1}\|_0^2 +
	\Qn{\nu \|\nabla \bxt\|_0^2 + S_h(\bxt,\bxt)} + \Qn{\|\bxt\|_0^2}\\
	\begin{aligned}
	& \le \|\bxt(t_{n-1}^-)\|_0^2
	+ C\, E_n^2 + 2\big(1+ C_n(\bu)\big)
	\Big(8 C_1 \tau_n\|\bxt (t_{n-1}^-)\|_0^2 + C E_n^2\Big)\\
	& \le \Big(
	1+ 16\big(1+C_n(\bu)\big)C_1\tau_n
	\Big) \|\bxt (t_{n-1}^-)\|_0^2
	+ C\big(1+C_n(\bu)\big) E_n^2 .
	\end{aligned}
	\end{multline*}
	The application of a discrete Gronwall lemma gives
	\begin{multline*}
	\|\bxt(t_n^-)\|_0^2  + 
	\sum_{m=1}^n \Big(Q_m\big[\nu \|\nabla \bxt\|_0^2 + S_h(\bxt,\bxt)
	+ \|\bxt\|_0^2\big] + \|[\bxt]_{m-1}\|_0^2\Big)\\
	\le C \exp\big(16C_1(1+C(\bu))t_n\big) (1+C(\bu)) \sum_{m=1}^n E_m^2
	\end{multline*}
	with $C(\bu)$ from~\eqref{eq:Cu}.
	The error splitting~\eqref{eq:error_splitting}, the triangle inequality,
	and the fact that $\bu$ and $\tu$ coincide in all quadrature points lead to
	\begin{multline*}
	\|\eht(t_n^-)\|_0^2  + 
	\sum_{m=1}^n
	\bigg(Q_m\big[\nu \|\nabla\eht\|_0^2
	+ S_h(\eht,\eht)
	+ \|\eht\|_0^2 \big] +  \|[\eht]_{m-1}\|_0^2\bigg) \\
	\begin{aligned}
	& \le 3 \|\bxt(t_n^-)\|_0^2  + 
	3\sum_{m=1}^n \Big(Q_m\big[\nu \|\nabla (\bxt) \|_0^2 + S_h(\bxt,\bxt)
	+\|\bxt\|_0^2\big] + \|[\bxt\|_{m-1}\|_0^2\Big)\\
	& \qquad +   3 \sum_{m=1}^n \Big(Q_m\big[\nu \|\nabla (j_h\bet)\|_0^2
	+ S_h(j_h\bet,j_h\bet) + \|j_h\bet\|_0^2\big] +
	\|[j_h\bet]_{m-1}\|_0^2\Big) \\
	&\qquad + 3 \|\beh(t_n^-)\|_0^2  + 
	3\sum_{m=1}^n \Big(Q_m\big[\nu \|\nabla \beh \|_0^2 + S_h(\beh,\beh)
	+ \|\beh\|_0^2 \big]\Big) .
	\end{aligned}
	\end{multline*}
	The statement of the theorem then follows by collecting the estimate for
	$\bxt$ as well as exploiting the stability and interpolation error
	estimates with respect to space and time. 
\end{proof}
\subsection{Pressure estimates: convergence}
This subsection will present a convergence result for the pressure that
depends unfortunately on the inverse of the length of the smallest time
step.
\begin{theorem}
	Suppose assumptions~\ref{assmption_a1}, \ref{assmption_a2},
	$\mu_K \sim 1$ for all $K\in\mathcal{T}_h$, and the discrete inf-sup
	condition~\eqref{eq:disc_inf_sup}. Furthermore, let the
	regularity assumption~\ref{reg:assumpt} hold. Then, for the solutions
	$(\uht,\pht)$ of the fully discrete scheme~\eqref{eq:fdp_one_form} and
	$(\bu,p)$ of the continuous problem~\eqref{eq:weak_form}, the error estimate
	\begin{equation}
	\label{eq:press_est}
	\int_0^T \|p-\pht\|_0^2 \le C \frac{1}{\tau_{\textrm{min}}^2} E^2
	\end{equation}
	holds true where $C$ is independent of $\nu$, $h$, and $\tau$ while
	\begin{equation}
	\label{eq:E}
	E^2 \coloneqq \exp\big(16C_1(1+C(\bu))T\big) (1+C(\bu))
	\sum_{m=1}^N E_m^2
	\end{equation}
	is the error bound for the velocity error.
\end{theorem}

\begin{proof}
	The proof of this lemma follows similar steps as in~\cite{ABM17}. However,
	the stabilization term and nonlinearity have to be taken into
	consideration.
	
	Consider $\tht=\pht-i_h\tp\in Y_k$ with the local representation
	\begin{equation*}
	\tht(t) =  \sum_{i=1}^{k+1} Q_{n,h}^i \varphi_{n,i}(t),\quad t\in I_n.
	\end{equation*}
	It follows from the discrete inf-sup condition~\eqref{eq:disc_inf_sup}
	that there exist discrete velocity fields $\wnh\in \vv_h$,
	$i=1,\ldots,k+1$, such that
	\begin{equation}\label{eq:pressure_inf_sup}
	\beta_0 \|\Qnh\|_0 \le \frac{(\Qnh,\Div \wnh)}{|\wnh|_1}.
	\end{equation}
	We obtain
	\[
	\Qnh = \pnh-i_h \tp(t_{n,i}) = \pnh - p(t_{n,i}) + p(t_{n,i}) -i_h p(t_{n,i})
	\]
	using $\tp(t_{n,i})=p(t_{n,i})$. Hence, we have 
	\begin{equation}\label{eq:pre_temp1}
	\big|(\Qnh,\Div \wnh)\big| \le \big|(\pnh - p(t_{n,i}), \Div \wnh)\big| +
	\big|\left(p(t_{n,i}) -i_h p(t_{n,i}), \Div \wnh\right)\big|.
	\end{equation} 
	We get from~\eqref{eq:weak_form} and~\eqref{eq:fdp_algebraic} that
	\begin{align*}
	\big(\pnh - p(t_{n,i}), \Div \wnh\big)
	&= -\frac{2\beta_i}{\tau_n}\left([\bu-\uht]_{n-1},\wnh  \right)
	- \left(\bu'(t_{n,i}) -\uht'(t_{n,i}),\wnh \right)
	- \nu \left(\nabla \enh, \nabla \wnh \right)\\
	& \qquad - n\big(\bu(t_{n,i}), \bu(t_{n,i}), \wnh\big)
	+ n \big( \unh, \unh, \wnh\big) + S_h(\unh, \wnh)
	\end{align*}
	where $\enh\coloneqq \bu(t_{n,i}) - \unh$.
	This expression is similar to the one for the transient Stokes
	problem considered in~\cite{ABM17}, but with the additional terms
	$n$ and $S_h$.
	
	It follows by Friedrichs and the Cauchy--Schwarz inequalities that
	\begin{multline}\label{eq:pressure_e1}
	\big(\pnh - p(t_{n,i}), \Div\wnh\big)\\
	\begin{aligned}
	& \le C \left[ \frac{2\beta_i}{\tau_n} \big\|[\bu-\uht]_{n-1}\big\|_0
	+ \big\|\bu'(t_{n,i}) -\uht'(t_{n,i}) \big\|_0
	+ \nu\big\|\nabla \enh\big\|_0\right] \left|\wnh\right|_1\\
	&\qquad+ \left|n\left(\bu(t_{n,i}), \bu(t_{n,i}), \wnh\right)-n\left(\unh, \unh, \wnh\right)\right|
	+ \left|S_h(\unh, \wnh)\right|.
	\end{aligned}
	\end{multline}
	The difference of the nonlinear terms is split as follows
	\begin{multline*}
	n\big(\bu(t_{n,i}),\bu(t_{n.i}),\wnh\big) - n\big(\unh,\unh,\wnh\big)\\
	\begin{aligned}
	& = n(\bu(t_{n,i}) - j_h \bu(t_{n,i}), \bu(t_{n,i}), \wnh)
	+ n(\bxt(t_{n,i}), j_h \bu(t_{n,i}), \wnh)\\
	& \qquad + n(j_h \bu(t_{n,i}),\bu(t_{n,i}) - j_h \bu(t_{n,i}),\wnh).
	\end{aligned}
	\end{multline*}
	We estimate term-by-term using generalised H\"older's and Friedrichs'
	inequalities. We get for the first term
	\begin{align*}
	n(\bu(t_{n,i}) - j_h \bu(t_{n,i}), \bu(t_{n,i}), \wnh)
	& \le \|\bu(t_{n,i}) - j_h \bu(t_{n,i})\|_0  \|\nabla \bu(t_{n,i})\|_{\infty}
	\|\wnh\|_0\\
	& \qquad + \frac{1}{2}\|\Div (\bu(t_{n,i}) - j_h \bu(t_{n,i}))\|_0
	\|\bu(t_{n,i})\|_{\infty} \|\wnh\|_0\\
	& \le C h^r \|\bu(t_{n,i})\|_{r+1,\infty} \|\bu(t_{n,i})\|_{1,\infty} |\wnh|_1
	\end{align*}
	where~\eqref{j1} was applied. The second term gives
	\begin{align*}
	n(\bxt(t_{n,i}), j_h \bu(t_{n,i}), \wnh)
	& \le \|\bxt(t_{n,i})\|_0 \|\nabla j_h \bu(t_{n,i})\|_{\infty} \|\wnh\|_0\\
	& \qquad + C_d S_h(\bxt(t_{n,i}),\bxt(t_{n,i}))^{1/2}
	\|j_h\bu(t_{n,i})\|_{\infty} \|\wnh\|_0\\
	& \le C \big( \|\bxt(t_{n,i})\|_0 + S_h(\bxt(t_{n,i}),\bxt(t_{n,i}))^{1/2} \big)
	\|\bu(t_{n,i})\|_{1,\infty} |\wnh|_1
	\end{align*}
	using Lemma~\ref{lem:discdiv} and~\eqref{jstab}. We obtain the estimate
	\begin{align*}
	n(j_h \bu(t_{n,i}),\bu(t_{n,i}) - j_h \bu(t_{n,i}),\wnh)
	& \le 
	\|j_h \bu(t_{n,i})\|_0
	\|\nabla (\bu(t_{n,i}) - j_h \bu(t_{n,i}))\|_{\infty}
	\|\wnh\|_0 \\
	& \qquad + \frac{1}{2} \|\Div j_h \bu(t_{n,i})\|_0 
	\|\bu(t_{n,i}) - j_h \bu(t_{n,i})\|_{\infty} \|\wnh\|_0\\
	& \le C h^r \|\bu(t_{n,i})\|_{r+1,\infty} \|\bu(t_{n,i})\|_{1,\infty}
	|\wnh|_1
	\end{align*}
	for the third term where~\eqref{j1} and~\eqref{jstab} have been used.
	
	For the last term on the right-hand side of \eqref{eq:pressure_e1}, the
	Cauchy--Schwarz inequality, the stability property of the fluctuation
	operator $\kappa_K$, and~\eqref{kappa} give
	\begin{align*}
	S_h\left(\unh,\wnh\right) &\le S_h\left(\unh,\unh\right)^{1/2}
	S_h\left(\wnh,\wnh\right)^{1/2} \le C\sqrt{\mu_h^{\textrm{max}}}
	S_h\left(\unh,\unh\right)^{1/2} \; |\wnh|_1\\
	& \le C\sqrt{\mu_h^{\textrm{max}}}
	\Big( S_h\left(\enh,\enh\right)^{1/2}
	+ S_h\big(\bu(t_{n,i}),\bu(t_{n,i})\big)^{1/2}\Big)
	|\wnh|_1\\
	& \le C \Big( S_h\left(\enh,\enh\right)^{1/2} + h^r \|\bu\|_{r+1,2} \Big)
	|\wnh|_1.
	\end{align*}
	Inserting all bounds together with~\eqref{eq:pre_temp1}
	into~\eqref{eq:pressure_inf_sup} leads to 
	\begin{align*}
	\|\Qnh\|_0 \le \frac{C}{\beta_0}
	\bigg[ & \big\|p(t_{n,i}) -i_h p(t_{n,i})\big\|_0
	+ \frac{2\beta_i}{\tau_n} \big\|[\bu-\uht]_{n-1}\big\|_0
	+ \big\|\bu'(t_{n,i}) -\uht'(t_{n,i}) \big\|_0 \\
	&+ \nu\big\|\nabla\enh\big\|_0
	+ h^r \|\bu(t_{n,i})\|_{r+1,\infty}
	\big(\|\bu(t_{n,i})\|_{1,\infty} + 1\big)\\
	& + \|\bxt(t_{n,i})\|_0 + S_h(\bxt(t_{n,i}),\bxt(t_{n,i}))^{1/2}
	+ S_h(\enh,\enh)^{1/2}
	\bigg].
	\end{align*}
	After squaring, multiplying by $\widehat{\omega}_i \tau_n/2$, and summing
	over $i=1,\ldots,k+1$, we get
	\begin{align}\label{eq:pre_temp2}
	\intin \| \tht \|_0^2 \le C \bigg[& \Qn{\|p -i_h p\|_0^2}
	+ \frac{1}{\tau_n} \big\|[\bu-\uht]_{n-1}\big\|_0^2
	+ \Qn{\|\bu' -\uht' \|_0^2} \nonumber\\
	&+ \Qn{\nu\|\nabla \eht\|_0^2}
	+ \tau_n h^{2r} \|\bu\|_{C(I_n,W^{r+1,\infty})}^2
	(\|\bu\|_{C(I_n,W^{1,\infty})}^2+1)
	\nonumber\\
	& + \Qn{\|\bxt\|_0^2} + \Qn{S_h(\bxt,\bxt)}
	+ \Qn{S_h(\eht,\eht)}
	\bigg].
	\end{align}
	We estimate the first term by the approximation properties~\eqref{j2} of
	$i_h$ and the second term by using
	Theorem~\ref{thm:fdp_velocity_estimates}.
	To estimate the third term in~\eqref{eq:pre_temp2},
	we proceed as follows
	\begin{align*}
	\Qn{\|\bu'-\uht'\|_0^2} & 
	\le 2 \Qn{\|\bu'-\tu'\|_0^2} + 2 \Qn{\|\tu'-\uht'\|_0^2} \le 2 \Qn{\|\bu'-\tu'\|_0^2}
	+ \frac{2C_{\mathrm{inv}}^2}{\tau_n^2} \Qn{\|\tu-\uht\|_0^2}\\
	& = 2 \Qn{\|\bu'-\tu'\|_0^2}
	+ \frac{4C_{\mathrm{inv}}^2}{\tau_n^2} \Qn{\|\bu-\uht\|_0^2}
	\end{align*}
	where an inverse inequality in time was applied in the second step.
	Furthermore, we exploited that $\bu$ and $\tu$ coincide in all quadrature
	points using by $Q_n$. The appearing terms can be bounded by the
	interpolation properties~\eqref{t3} and the estimate from
	Theorem~\ref{thm:fdp_velocity_estimates}.
	
	The remaining terms in~\eqref{eq:pre_temp2} can be estimated by using
	again Theorem~\ref{thm:fdp_velocity_estimates}. We end up with
	\[
	\int_0^T \|p-\pht\|_0^2 \le C \frac{1}{\tau_{\textrm{min}}^2} E^2
	\]
	with $E$ given in~\eqref{eq:E}.
\end{proof}

\section{Numerical studies}\label{sec:numerics}
This section presents the numerical studies to verify the theoretical
predictions of the previous sections. For this purpose, we consider the two
different examples. In the first example a problem will be studied where
the spatial error dominates. With this example, the order of convergence in
space can be assessed in different norms. The second example where the
temporal error dominates will show the convergence order in time.

We choose $T=1$ as final time while the computational domain for both
examples is $\Omega=(0,1)^2$ and the
simulations were performed on uniform quadrilateral grids where the
coarsest grid (level 1) is obtained by dividing the unit square into four
congruent squares. We used in our numerical simulations mapped finite
element spaces, see~\cite{Cia78}, where the enriched spaces on the
reference cell $\widehat{K}=(-1,1)^2$ are given by 
\[
\mathbb{Q}_r^{ \mathrm{bubble}}(\widehat{K}) := \mathbb{Q}_r + \mathrm{span}\left\{(1-\hat{x}_1^2)(1-\hat{x}_2^k)\hat{x}_i^{r-1}, \; i=1,2 
\right\}.
\]
The combination $\mathbb{Q}_r^{\mathrm bubble}(\hat{K})$ with $D(K) =
\mathbb{P}_{r-1}(K)$ provides for $r\ge 2$ suitable spaces for LPS methods,
see~\cite{MT15}. The stabilization parameter for the dominant convection
case is set to $\mu_K=0.1$.

We will use the norm \(\|\eht\|_{\mathrm{S}} \) that is the combinations of the 
terms of the left-hand side of \eqref{eq:velo_est}
\[
\|\eht\|_{\mathrm{S}} = \Big(\|\eht(t_n^-)\|_0^2
+  \|[\eht]_{n-1}\|_0^2 +
\Qn{\nu \|\nabla \eht\|_0^2 + S_h(\eht,\eht)} + \Qn{\|\eht\|_0^2}\Big)^{1/2}.
\]
\subsection{Example with dominating space error}
We consider the first example where the time error is negligible. The right-hand side $\ff$ and the initial condition $\bu_0$ are chosen such that
\begin{align*}
\bu(t,x,y) &= \sin(t) \Big(\sin(\pi x) \sin(\pi y), \cos(\pi x) \cos(\pi y) \Big)^T \\
p(t,x,y) &= \sin(t) \left(\sin(\pi x) + \cos(\pi y) - \frac2\pi \right)
\end{align*}
is the solution of \eqref{eq:nse} equipped with the non-homogeneous Dirichlet boundary 
conditions. 

To illustrate the convergence order in space, we performed the numerical
simulations using the time discretization scheme dG(1) with a small time
step $\tau=1/800$. Figure~\ref{fig:space_errors} presents the convergence
results for the simulations performed with the finite element spaces
$\vv_h/Q_h=\mathbb{Q}_2^{\mathrm{bubble}}/\mathbb{P}_1^{\mathrm{disc}}$ and 
the projection space $D_h(K)=\mathbb{P}_1(K)$ for the LPS method and
$\vv_h/Q_h=\mathbb{Q}_2/\mathbb{P}_1^{\mathrm{disc}}$ for the standard finite element method. One can clearly see from the plots that the corresponding convergence orders are obtained in all norm as predicted in \eqref{eq:velo_est}.
\begin{figure}[htb!]
\begin{center}
\includegraphics[scale=1]{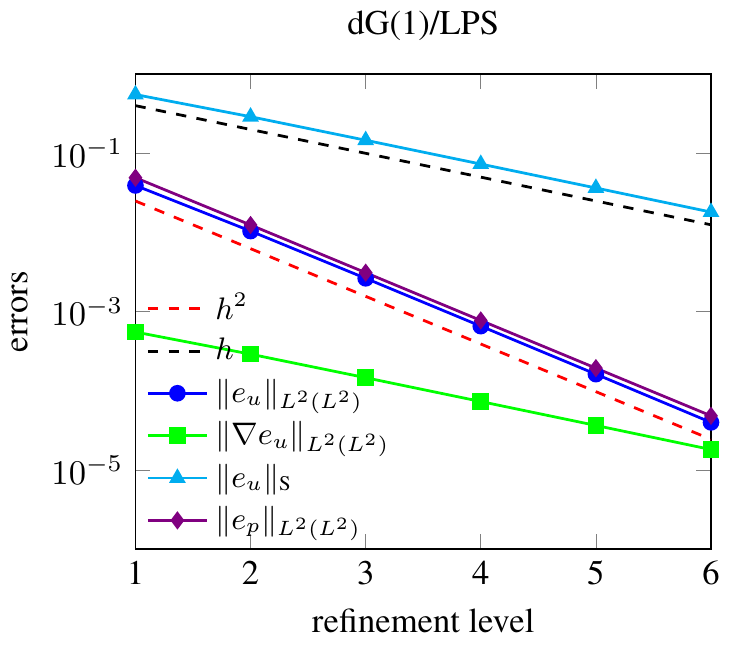}~
\includegraphics[scale=1]{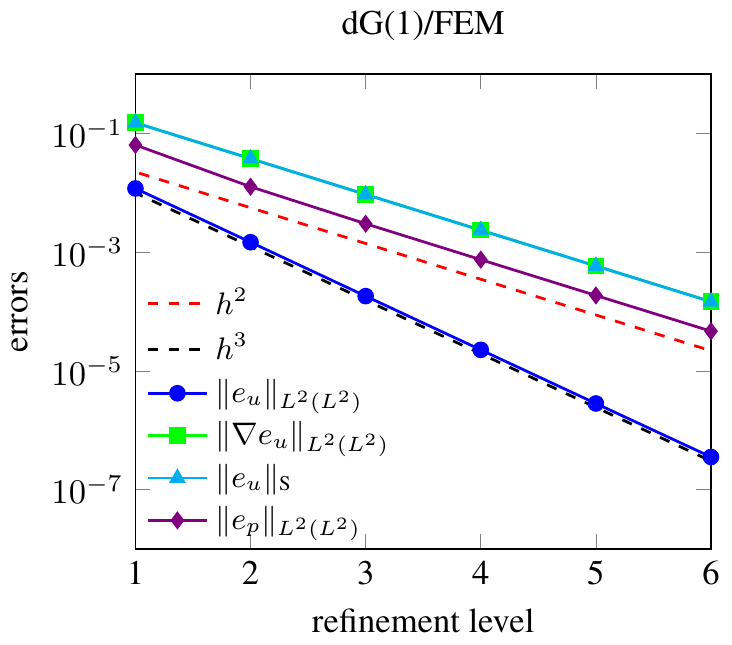}
\caption{Convergence orders with respect to the spatial mesh width, 
		$\nu=10^{-6}$,
		$\vv_h/Q_h=\mathbb{Q}_2^{\mathrm{bubble}}/\mathbb{P}_1^{\mathrm{disc}}$ (left), and 
		$\nu=1 $, $\vv_h/Q_h=\mathbb{Q}_2/\mathbb{P}_1^{\mathrm{disc}}$ (right).}\label{fig:space_errors}
		\end{center}
\end{figure}

\subsection{Example with dominating time error}
This example studies the convergence orders in time. The right-hand side $\ff$ and the initial condition $\bu_0$ are chosen such that
\begin{align*}
u_1(t,x,y) &= x^2 (1-x)^2 \left(2y(1-y)^2 - 2y^2(1-y)\right) \; \sin(10\pi t),\\
u_2(t,x,y) &= -\left(2x(1-x)^2 - 2x^2(1-x)\right)y^2(1-y)^2\; \sin(10\pi t), \\
p(t,x,y)   &= -(x^3+y^3-0.5)\left(1.5+0.5\sin(10\pi t)\right)
\end{align*}  
is the solution of \eqref{eq:nse} with homogeneous Dirichlet boundary conditions.

In order to study the convergence order in time, the simulations were
performed with
$\vv_h/Q_h=\mathbb{Q}_4^{\mathrm{bubble}}/\mathbb{P}_3^{\mathrm{disc}}$ and
$D(K)=\mathbb{P}_3(K)$ for all $K\in \mathcal{T}_h$ and a mesh which consists of $8\times 8$ squares. The calculations were done for dG(1) and dG(2) with the time step lengths $\tau=0.1\times 10^i,\; i=0,\ldots,6$.

Figure~\ref{fig:time_errors_dg1} report the convergence order for the methods dG($k$), $k=1,2$ in combination with the spatial stabilization by LPS. The errors in different norms are plotted against the different refinement levels in time. 
It can be seen that the dG($k$) method is accurate of order $k+1$ in the $L^2(L^2)$-norm while the order $k+1/2$ is observed in the $\|\cdot\|_{\mathrm{S}}$-norm. These results are in agreement with the theoretical 
predictions in Theorem~\ref{thm:fdp_velocity_estimates}. 

Comparing the convergence order for the pressure in $L^2(L^2)$-norm, one can see that the convergence is one order better than predicted by the theory~\eqref{eq:press_est}. This is 
caused by the smoothness of the pressure. However, if we consider the problem where the pressure is replaced by the rough function
\[
p(t,x,y)=-(x^3+y^3-0.5)\left(1.5+0.5t^{4/3}\right)
\]
then one can see in figure~\ref{fig:rough_pressure}, the convergence order for the pressure is limited by the smoothness. On the other hand, the convergence order of the velocity are not influenced by the smoothness of the pressure.

A simple post-processing of the time-discrete solution $\bu_{h,\tau}$ allows to obtain numerical approximations which are order $k+2$ in the integral based norms, see~\cite{ABM17,MS11}. 
The result for the post-processed solution are presented in figure~\ref{fig:time_errors_pp}. One can see the improved accuracy in the $L^2(L^2)$-norm and $\|\cdot\|_\mathrm{ S}$-norm.
\begin{figure}[htb!]
\begin{center}
    \includegraphics[scale=1]{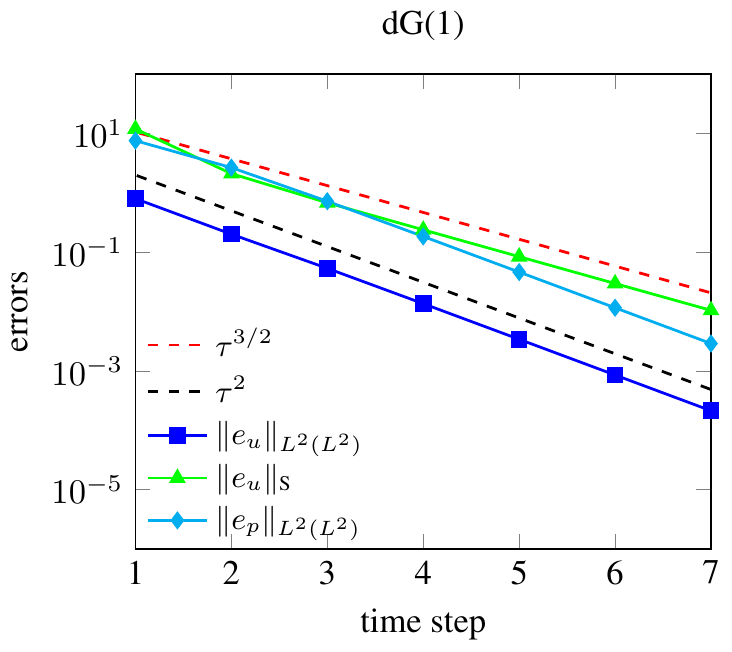}~
    \includegraphics[scale=1]{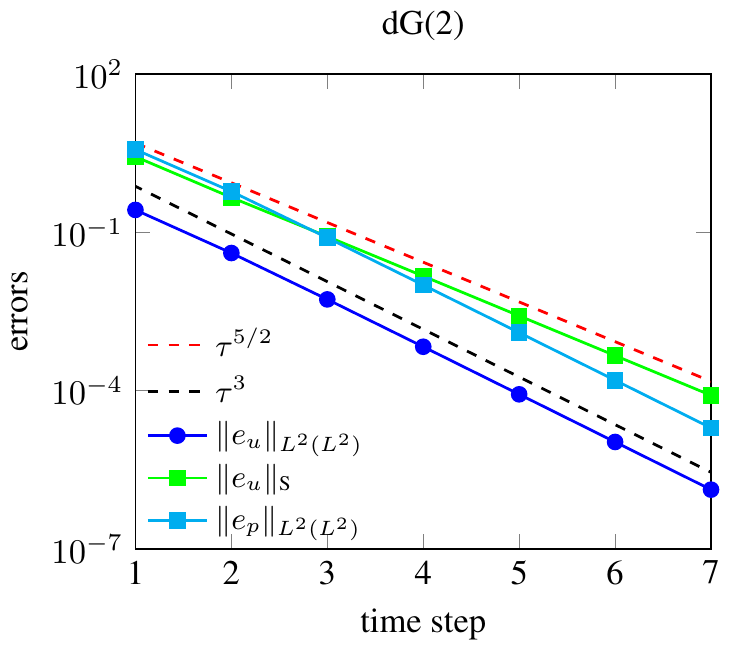}
	\caption{Convergence orders for different errors of the solution for dG(1) and dG(2).}\label{fig:time_errors_dg1}
	\end{center}
\end{figure}
\begin{figure}[htb!]
\begin{center}%
	\includegraphics[scale=1]{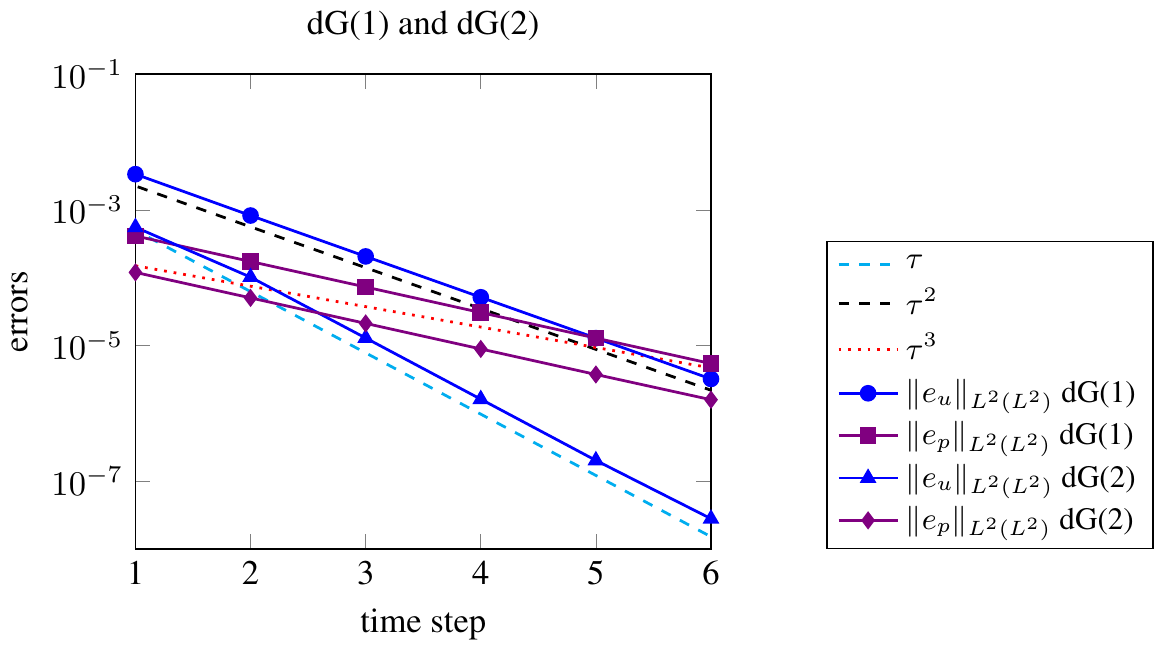}
	\caption{Example with rough pressure: convergence orders for dG(1) and dG(2) methods.}\label{fig:rough_pressure}
	\end{center}
\end{figure}

\begin{figure}[htb!]
\begin{center}
	\includegraphics[scale=1]{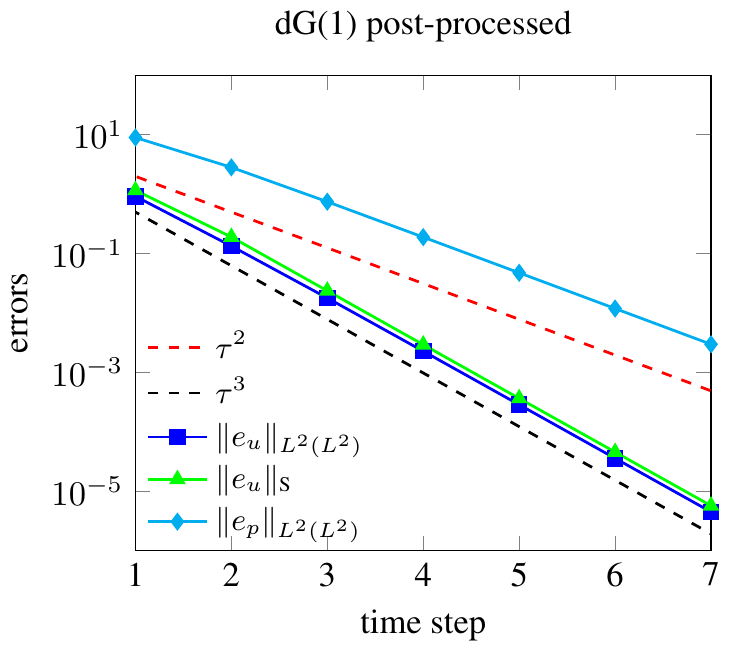}~
	\includegraphics[scale=1]{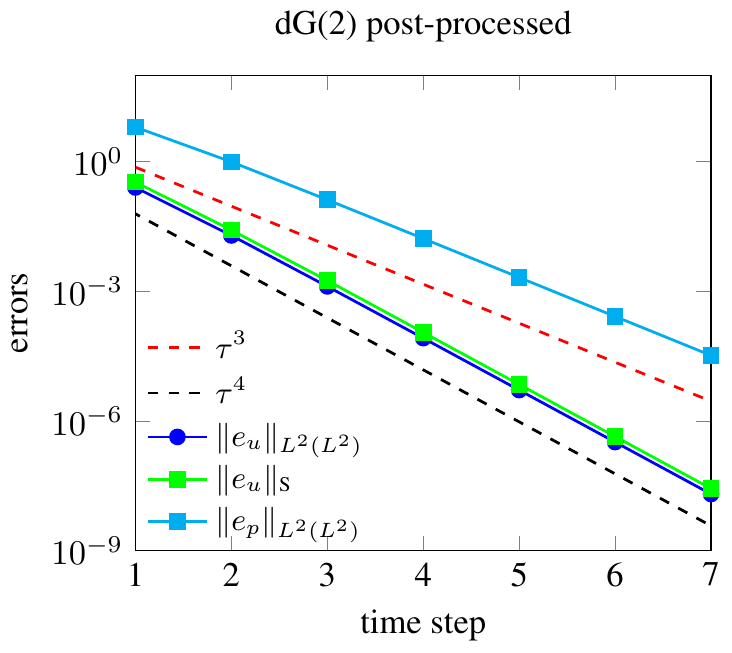}
	\caption{Convergence orders for different errors of the post-processed solution for dG(1) and dG(2) methods.}
	\label{fig:time_errors_pp}
\end{center}
\end{figure}


\end{document}